\documentclass[11pt,a4paper]{article}
\usepackage{amsmath}
\usepackage{amssymb}
\usepackage{amsthm}
\usepackage{amsfonts}
\usepackage{latexsym}
\usepackage{verbatim} 
\usepackage{xcolor}
\usepackage[a4paper,width=16cm,top=2cm,bottom=2cm,footskip=1cm
]{geometry}

\usepackage[flushmargin]{footmisc}
\usepackage{color}                    
\usepackage{esvect}
\usepackage{mathrsfs}

\theoremstyle{plain}
\newtheorem{thm}{Theorem}[section]
\newtheorem{cor}{Corollary}[section]
\newtheorem{lem}{Lemma}[section]

\newtheorem{ass}{Assumption}[section]
\theoremstyle{remark}

\theoremstyle{definition}
\newtheorem{defn}{Definition}[section]

\newtheorem{rem}{Remark}[section]

\newcommand{\Complex}{\mathbb C}
\newcommand{\Real}{\mathbb R}
\newcommand{\N}{\mathbb N}
\newcommand{\ddbar}{\overline\partial}
\newcommand{\pr}{\partial}
\newcommand{\ol}{\overline}
\newcommand{\Td}{\widetilde}
\newcommand{\norm}[1]{\left\Vert#1\right\Vert}
\newcommand{\abs}[1]{\left\vert#1\right\vert}
\newcommand{\set}[1]{\left\{#1\right\}}
\newcommand{\To}{\rightarrow}

\title{Toeplitz operators on CR manifolds and group actions}
\author{Andrea Galasso and Chin-Yu Hsiao\footnote{\noindent{\bf Address:} Room 407, Chee-Chun Leung Cosmology Hall, National Taiwan University; {\bf ORCID iD:} 0000-0002-5792-1674; {\bf e-mail}: andrea.galasso@ncts.ntu.edu.tw {\bf Address:} Institute of Mathematics, Academia Sinica, 6F, Astronomy-Mathematics Building, No.1, Sec.4, Roosevel: {\bf ORCHID iD:} 0000-0002-1781-0013; {\bf email}: chsiao@math.sinica.edu.tw; chinyu.hsiao@gmail.com}}

\date{}

\begin{document}
\maketitle

\begin{abstract}  Let $(X,\,T^{1,0}X)$ be a connected orientable compact CR manifold of dimension $2n+1,\, n\geq 1$ with non-degenerate Levi curvature. 
In this paper, we study the algebra of  Toeplitz operators on $X$ and we establish star product for some class of symbols on $X$. In the second part of this paper, we consider a compact locally free Lie group $G$ acting on $X$ and we investigate the associated algebra of $G$-invariant Toeplitz operators.  
\end{abstract}
\tableofcontents


\section{Introduction}

Let $(X,\,T^{1,0}X)$ be a connected orientable compact CR manifold of dimension $2n+1,\, n\geq 1$. Let $\Box^{q}_b$ be the K\"ohn Laplacian acting on $(0,\,q)$ forms. The corresponding orthogonal projection $S^{(q)}\,:\,L^2_{(0,q)}(X)\rightarrow \ker\Box^{q}_b$ is called the Szeg\H{o} projector and its distributional kernel $S^{(q)}(x,y)$ was studied in \cite{hsiao}; when $X$ is strongly pseudo-convex of dimension greater or equal to five, Boutet de Monvel and Sj\"ostrand \cite{bs} proved that $S^{(0)}$ is a Fourier integral operator of complex type. 
Given a pseudo-differential operator $P$, one can define a Toeplitz operator $T_P:=S^{(q)}\circ P \circ S^{(q)}$. The algebra of Toeplitz operators on a compact strongly pseudoconvex CR manifold was originally investigated by Boutet de Monvel and Guillemin \cite{bg}. When $X$ is not strongly pseudoconvex, there are fewer results. In this work, we consider the case when the Levi form of $X$ is non-degenerate of constant signature $(n_-, n_+)$ and we study the algebra of $T_P:=S^{(q)}\circ P \circ S^{(q)}$, where $q=n_-$ and $P$ is a pseudo-differential operator with scalar principal symbol. In the second part of this paper, we consider a compact Lie group $G$ acting on $X$. The first aim of the second part is to investigate the associated algebra of $G$-invariant Toeplitz operators, secondly we study the associated Fourier component in the presence of a transversal locally free circle action. 

We now formulate the main results. 
We refer the reader to Section~\ref{s:prelim} for some notations and terminology used here. Let $(X,\,T^{1,0}X)$ be a connected orientable compact CR manifold of dimension $2n+1,\, n\geq 1$. Fix a Reeb one form $\omega_0\in\mathcal{C}^\infty(X,T^*X)$ and let $R\in\mathcal{C}^\infty(X,TX)$ be the Reeb vector field given by \eqref{e-gue170111ry}. For every $x\in X$, let $\mathcal{L}_x$ be the Levi form of $X$ at $x\in X$ associated to $\omega_0$ (see \eqref{e-gue210813yyd}). Fix a smooth Hermitian metric $\langle\, \cdot \,|\, \cdot \,\rangle$ on $\mathbb{C}TX$ so that $T^{1,0}X$ is orthogonal to $T^{0,1}X$,  $\langle\,R\,|\,R\,\rangle=1$ and $R$ is orthogonal to $T^{1,0}X\oplus T^{0,1}X$. The Hermitian metric $\langle\, \cdot\, |\, \cdot \,\rangle$ on $\mathbb{C}TX$ induces, by duality, a Hermitian metric on $\mathbb{C}T^*X$ and also on the bundles of $(0,q)$ forms $T^{*0,q}X, q=0, 1, \cdots, n$. We shall also denote all these induced metrics by $\langle \,\cdot \,|\, \cdot \,\rangle$. Let $dv(x)$ be the volume form on $X$ induced by the Hermitian metric $\langle\,\cdot\,|\,\cdot\,\rangle$ and let 
 $(\,\cdot\,|\,\cdot\,)$ be the $L^2$ inner product on $\Omega^{0,q}(X)$ 
induced by $dv(x)$ and $\langle\,\cdot\,|\,\cdot\,\rangle$. Let $\Box^{q}_{b}$ denote the (Gaffney extension) of the K\"ohn Laplacian given by \eqref{e-suIX}. 
The characteristic set of $\Box^q_b$ is given by 
\begin{equation}\label{e-gue210813yydI}
\begin{split}
&\Sigma=\Sigma^-\bigcup\Sigma^+,\\
&\Sigma^-=\set{(x,\lambda\omega_0(x))\in T^*X;\,\lambda<0},\\
&\Sigma^+=\set{(x,\lambda\omega_0(x))\in T^*X;\,\lambda>0}.
\end{split}
\end{equation}
In this work, we assume that 

\begin{ass}\label{a-gue210813yydII}
The Levi form is non-degenerate of constant signature $(n_-,n_+)$, where $n_-$ denotes the 
number of negative eigenvalues of the Levi form and $n_+$ denotes the 
number of positive eigenvalues of the Levi form. We always let $q=n_-$ and suppose that 
$\Box^q_b$ has $L^2$ closed range. 
\end{ass}

It should be mention that if $q=n_-=n_+$ or $\abs{n_--n_+}>1$, then $\Box^q_b$ 
has $L^2$ closed range (see~\cite{Koh86}).

Let 
\[S^{(q)}: L^2_{(0,q)}(X)\To{\rm Ker\,}\Box^q_b\]
be the orthogonal projection (Szeg\H{o} projector). 
It is known that (see~\cite[Theorem 1.2]{hsiao},~\cite[Theorem 4.7]{HM14}) there exist continuous operators
$S_-, S_+: L^2_{(0,q)}(X)\To{\rm Ker\,}\Box^q_b$ such that 
\begin{equation}\label{e-gue210813yydIII}
\begin{split}
&S^{(q)}=S_-+S_+,\\
&{\rm WF'\,}(S_-)={\rm diag\,}(\Sigma^-\times\Sigma^-),\\
&S_+\equiv0\ \ \mbox{if $q\neq n_+$},\\
&{\rm WF'\,}(S_+)={\rm diag\,}(\Sigma^+\times\Sigma^+)\ \ \mbox{if $q=n_-=n_+$},
\end{split}
\end{equation}
where ${\rm WF'\,}(S_-)=\set{(x,\xi,y,\eta)\in T^*X\times T^*X;\,(x,\xi,y,-\eta)\in{\rm WF\,}(S_-)}$, ${\rm WF\,}(S_-)$ is the wave front set of $S_-$ in the sense of H\"ormander. For $m\in\mathbb R$, 
let $L^m_{{\rm cl\,}}(X,T^{*0,q}X\boxtimes(T^{*0,q}X)^*)$ denote the space of classical 
pseudodifferential operators on $D$ of order $m$ from sections of $T^{*0,q}X$ to 
sections of $T^{*0,q}X$.
Let $P\in L^\ell_{{\rm cl\,}}(X,T^{*0,q}X\boxtimes(T^{*0,q}X)^*)$ with scalar principal symbol, $\ell\leq0$, $\ell\in\mathbb Z$.
The Toeplitz operator is given by 
\begin{equation}\label{e-gue210813yydc}
\begin{split}
&T_P:=S^{(q)}\circ P\circ S^{(q)}: L^2_{(0,q)}(X)\To{\rm Ker\,}\Box^q_b,\\
&T_{P,-}:=S_-\circ P\circ S_-: L^2_{(0,q)}(X)\To{\rm Ker\,}\Box^q_b,\\
&T_{P,+}:=S_+\circ P\circ S_+: L^2_{(0,q)}(X)\To{\rm Ker\,}\Box^q_b.
\end{split}
\end{equation}
If $q\neq n_+$, then $T_{P,+}\equiv0$ on $X$. 
Let $f\in\mathcal{C}^\infty(X)$ and let $M_f$ denote the operator given by the multiplication $f$. When $P=M_f$, we write $T_f:=T_P$. For every $k\in\mathbb Z$, let $\Psi_k(X)$ denote the space of all complex Fourier integral operators of Szeg\H{o} type of order $k$ (see Definition~\ref{d-gue210624yyd}). The first main result is about the algebra of Toeplitz operators.

\begin{thm}\label{t-gue210701yydz}
Let $P\in L^{\ell}_{{\rm cl\,}}(X,T^{*0,q}X\boxtimes(T^{*0,q}X)^*)$, $Q\in L^{k}_{{\rm cl\,}}(X,T^{*0,q}X\boxtimes(T^{*0,q}X)^*)$ with scalar principal symbols, $\ell, k\leq0$, $\ell, k\in\mathbb Z$. 
We have $[T_{P,\mp},T_{Q,\mp}]\in\Psi_{n-1+\ell+k}(X)$ and 
\begin{equation}\label{e-gue210813yyda}\begin{split}
&[T_{P,-}, T_{Q,-}]-T_{{\rm Op\,}(\imath\{\sigma^0_P,\sigma^0_Q\}_-),-} \in\Psi_{n-2+\ell+k}(X),\\
&[T_{P,+}, T_{Q,+}]-T_{{\rm Op\,}(\imath\{\sigma^0_P,\sigma^0_Q\}_+),+} \in\Psi_{n-2+\ell+k}(X)\ \ \mbox{if $q=n_-=n_+$},
\end{split}\end{equation}
where $\{\sigma^0_P,\sigma^0_Q\}_{\mp}$ denotes the  transversal Poisson bracket  of $\sigma^0_P$ and $\sigma^0_Q$ with respect to $\Sigma^{\mp}$ (see \eqref{e-gue210721yyda}), 
$\sigma^0_P$ and $\sigma^0_Q$ are principal symbols of $P$ and $Q$ respectively. 
\end{thm}

\begin{cor}\label{c-gue210813yyd}
Let $f, g\in\mathcal{C}^\infty(X)$. We have $[T_{f,\mp},T_{g,\mp}]\in\Psi_{n-1}(X)$ and 
\begin{equation}\label{e-gue210813yydb}\begin{split}
&[T_{f,-}, T_{g,-}]-T_{\imath\{f,g\},-} \in\Psi_{n-2}(X),\\
&[T_{f,+}, T_{g,+}]-T_{\imath\{f,g\},+} \in\Psi_{n-2}(X)\ \ \mbox{if $q=n_-=n_+$},
\end{split}\end{equation}
where $\{f,g\}$ denotes the  transversal Poisson bracket  of $f$ and $g$  (see \eqref{e-gue210706yydI}).  
\end{cor}

To establish star product on $X$, we introduce more notations. Let 
\begin{equation}\label{e-gue210722yydq}
\triangle_b:=\Box^q_b+R^*R: \Omega^{0,q}(X)\To\Omega^{0,q}(X),
\end{equation}
where $R^*$ is the formal adjoint of the Reeb vector field $R$ with respect to $(\,\cdot\,|\,\cdot\,)$. For every $j\in\mathbb Z$, $j\leq0$, put 
\begin{equation}\label{e-gue210722yydIq}
\hat S^j\:=\set{f(x)(\sigma^0_{\triangle_b})^{\frac{j}{2}}\in\mathcal{C}^\infty(TX);\, f(x)\in\mathcal{C}^\infty(X)}
\end{equation}
and let 
\begin{equation}\label{e-gue210722yydIIq}
\hat S:=\bigcup_{j\in\mathbb Z, j\leq0}\hat S^j, 
\end{equation}
where $\sigma^0_{\triangle_b}$ denotes the principal symbol of $\triangle_b$. 
Define
\begin{equation}\label{e-gue210724yydaq}
a\hat{+}b:=\mbox{leading term of $a+b$}. 
\end{equation}
For $a, b\in\hat S$, then $a\hat{+}b\in\hat S^j$, for some $j\in\mathbb Z$, $j\leq0$. We can identify $\hat S^j$ with all homogeneous functions on $\Sigma$ of degree $j$. 
Then, $\set{\hat S,\hat{+}}$ is a vector space and also $\hat S$ has natural algebraic structure, that is, if $a, b\in\hat S$, then $a.b\in\hat S$. 
A star product for the algebra of $\hat S$ with respect to $\Sigma^{\mp}$ is given by the power series \[a*b= \sum_{j=0}^{+\infty} C_{j,\mp}(a,b) \nu^{-j}  \]
such that $*$ is an associative $\mathbb C[[\nu]]$-linear product, that is, $(a*b)*c=a*(b*c)$, for all $a, b, c\in\hat S$ and $C_{0,\mp}(a,b)=a\cdot b$, $C_{1,\mp}(a,b)-C_{1,\mp}(b,a)=\imath\{a, b\}_{\mp}$, for all $a, b\in\hat S$. To simplify the notations, for $a\in\hat S$, we denote 
\begin{equation}\label{e-gue210724ycdq}
T_a:=T_{{\rm Op\,}(a)},\ \ T_{a,-}:=T_{{\rm Op\,}(a),-},\ \ T_{a,+}:=T_{{\rm Op\,}(a),+}.
\end{equation}

Our second main result is about the existence of star product for $\hat S^j$.

\begin{thm}\label{t-gue210724yydaq}
Let $a\in\hat S^\ell$, $b\in\hat S^k$, $\ell, k\in\mathbb Z$, $\ell, k\leq0$. We have
\begin{equation}\label{e-gue210724ycdIq}
T_{a,-}\circ T_{b,-}-\sum^N_{j=0}T_{C_{j,-}(a,b),-}\in\Psi_{n-N-1+\ell+k}(X),
\end{equation}
for every $N\in\mathbb N_0$, where $C_{j,-}(a,b)\in\hat S^{\ell+k-j}$, $C_{j,-}$ is a universal bi-differential operator of order $\leq 2j$, $j=0,1,\ldots$, and

\begin{equation}\label{e-gue210724ycdIIq}
\begin{split}
C_{0,-}(a,b)=ab,\\
C_{1,-}(a,b)-C_{1,-}(b,a)=\imath\{a,b\}_-,
\end{split}
\end{equation} 
where $\{a,b\}_{-}$ denotes the  transversal Poisson bracket  of $a$ and $b$ with respect to $\Sigma^{-}$ (see \eqref{e-gue210721yyda}). 

Moreover, The star product
 \[a*b= \sum_{j=0}^{+\infty} C_{j,-}(a,b) \nu^{-j} ,\]
 $a, b\in\hat S$, is associative. 
 \end{thm} 
 
When $X$ admits a transversal and CR $\mathbb R$-action, we establish star product for $\mathbb R$-invariant smooth functions (see Theorem~\ref{t-gue210706ycdg} and Theorem~\ref{t-gue210725yyd}). For the better understanding, we state our results in the situation of complex manifolds. Let $(L,h^L)$ be a holomorphic line bundle over a compact complex manifold $M$ and let $(L^k,h^{L^k})$ be the $k$-th power of $(L,h^L)$, where $h^L$ denotes the Hermitian metric of $L$. Let $R^L$ be the curvature of $L$ induced by $h^L$. Fix a Hermitian metric $\langle\,\cdot\,|\,\cdot\,\rangle$ on the holomorphic tangent bundle $T^{1,0}M$ of $M$ and let $(\,\cdot\,|\,\cdot)_k$ be the $L^2$ inner product of $\Omega^{0,q}(M,L^k)$ induced by $\langle\,\cdot\,|\,\cdot\,\rangle$  and $h^{L^k}$, where $\Omega^{0,q}(M,L^k)$ denotes the space of smooth $(0,q)$ forms of $M$ with values in $L^k$. Let 
\[\Box^q_k:=\ddbar^*\,\ddbar+\ddbar\,\ddbar^*: \Omega^{0,q}(M,L^k)\To \Omega^{0,q}(M,L^k)\]
be the Kodaira Laplacian, where $\ddbar^*$ is the adjoint of $\ddbar$ with respect to $(\,\cdot\,|\,\cdot)_k$. Let 
\[B^{(q)}_k: L^2_{(0,q)}(M,L^k)\To{\rm Ker\,}\Box^q_k\]
be the orthogonal projection with respect to $(\,\cdot\,|\,\cdot)_k$ (Bergman projection). For $f\in\mathcal{C}^\infty(M)$, let $M_f$ denote the operator given by the multiplication $f$. The Toeplitz operator is given by 
\[T_{f,k}:=B^{(q)}_k\circ M_f\circ B^{(q)}_k: L^2_{(0,q)}(M,L^k)\To{\rm Ker\,}\Box^q_k.\]
 Applying Theorem~\ref{t-gue210706ycdg} and Theorem~\ref{t-gue210725yyd} to the circle bundle of $(L^*, h^{L^*})$ and by using the proof of Theorem~\ref{thm:compositionFouriercomponents}, we get 
 
\begin{thm} \label{t-gue210813yydr}
	With the same assumptions as above, suppose that the curvature $R^L$ is non-degenerate of constant signature $(n_-,n_+)$. Let $q=n_-$. Let $f, g\in\mathcal{C}^\infty(M)$. Then, as $k\gg1$, 
	\begin{equation}\label{e-gue210803ycdsz}
	\norm{T_{f,k}\circ T_{g,k}-T_{g,k}\circ T_{f,k}-\frac{1}{k}T_{\imath\{f,g\},k}}=O(k^{-2}),\end{equation}
	and 
\begin{equation}\label{e-gue210803ycdtz}
\norm{T_{f,k}\circ T_{g,k}-\sum^N_{j=0}k^{-j}T_{C_j(f,g),k}}=O(k^{-N-1})
\end{equation}
in $L^2$  operator norm, for every $N\in\mathbb N$, 
where $C_j(f,g)\in\mathcal{C}^\infty(M)$, $C_j$ is a universal bidifferential operator of order $\leq 2j$, $j=0,1,\ldots$, and 
\begin{equation}\label{e-gue210803ycduz}
\begin{split}
C_{0}(f,g)=fg,\\
C_{1}(f,g)-C_{1}(g,f)=\imath\{f,g\}.
\end{split}
\end{equation}

 Moreover, the star product
 \begin{equation}\label{e-gue210805yydz}
 f*g= \sum_{j=0}^{+\infty} C_{j}(f,g) \nu^{-j} ,
\end{equation}
 $f, g\in\mathcal{C}^\infty(M)$, is associative.
\end{thm} 

This Theorem can be thought of as an analogue of Theorem $2.2$ in \cite{schli} in the more general setting of complex manifolds. We remark that the spaces of $(0,q)$-forms we are considering here can be identified with the kernel of the Dolbeault-Dirac operator, see equation~$(2.36)$, pag.~$16$ in~\cite{Du}. In general it is known that there exists a deformation for the algebra of smooth functions on $M$ in terms of a formal power series. The advantage of defining the star product in terms of Toeplitz operators is that we get the convergence of the corresponding power series.

In Section~\ref{s-gue210812yyd}, we consider a compact Lie group $G$ acting on $X$. We generalize the results  in \cite{hsiaohuang} to Toeplitz operators and to the action of $G$ on $X$ is locally free. We then investigate  the associated algebra of $G$-invariant Toeplitz operators. Finally, we study the associated Fourier component in the presence of a transversal locally free circle action. We refer the reader to Theorems~\ref{thm:toeplitz},~\ref{thm:composition},~\ref{t-gue210803ycd},~\ref{thm:kfourierszego}, \ref{thm:compositionFouriercomponents}, for the details. 

\section{Preliminaries}\label{s:prelim}

\subsection{Standard notations} \label{s-ssna}
We use the following notations through this article: $\mathbb N=\{1,2,\ldots\}$ is the set of natural numbers, $\mathbb N_0=\mathbb N\bigcup\{0\}$, $\mathbb R$ is the set of 
real numbers, $\overline{\mathbb R}=\{x\in\mathbb R;\, x>0\}$, $\overline{\mathbb R}_+=\{x\in\mathbb R;\, x\geq0\}$. 
We write $\alpha=(\alpha_1,\ldots,\alpha_n)\in\mathbb N^n_0$ 
if $\alpha_j\in\mathbb N_0$, 
$j=1,\ldots,n$. 

Let $M$ be a $\mathcal{C}^\infty$ paracompact manifold.
We let $TM$ and $T^*M$ denote the tangent bundle of $M$
and the cotangent bundle of $M$, respectively.
The complexified tangent bundle of $M$ and the complexified cotangent bundle of $M$ will be denoted by $\Complex TM$
and $\Complex T^*M$, respectively. Write $\langle\,\cdot\,,\cdot\,\rangle$ to denote the pointwise
duality between $TM$ and $T^*M$.
We extend $\langle\,\cdot\,,\cdot\,\rangle$ bilinearly to $\Complex TM\times\Complex T^*M$.
Let $B$ be a $C^\infty$ vector bundle over $M$. The fiber of $B$ at $x\in M$ will be denoted by $B_x$.
Let $E$ be a vector bundle over a $\mathcal{C}^\infty$ paracompact manifold $M_1$. We write
$B\boxtimes E^*$ to denote the vector bundle over $M\times M_1$ with fiber over $(x, y)\in M\times M_1$
consisting of the linear maps from $E_y$ to $B_x$.  Let $Y\subset M$ be an open set. 
From now on, the spaces of distribution sections of $B$ over $Y$ and
smooth sections of $B$ over $Y$ will be denoted by $\mathcal D'(Y, B)$ and $\mathcal{C}^\infty(Y, B)$, respectively.
Let $\mathcal E'(Y, B)$ be the subspace of $\mathcal D'(Y, B)$ whose elements have compact support in $Y$.

We recall the Schwartz kernel theorem.
Let $B$ and $E$ be $\mathcal{C}^\infty$ vector
bundles over paracompact orientable $\mathcal{C}^\infty$ manifolds $M$ and $M_1$, respectively, equipped with smooth densities of integration. If
$A: \mathcal{C}^\infty_0(M_1,E)\To\mathcal D'(M,B)$
is continuous, we write $A(x, y)$ to denote the distribution kernel of $A$.
The following two statements are equivalent
\begin{enumerate}
\item $A$ is continuous: $\mathcal E'(M_1,E)\To\mathcal{C}^\infty(M,B)$,
\item $A(x,y)\in\mathcal{C}^\infty(M\times M_1,B\boxtimes E^*)$.
\end{enumerate}
If $A$ satisfies (1) or (2), we say that $A$ is smoothing on $M \times M_1$. Let
$A,\hat A:\mathcal{C}^\infty_0(M_1,E)\To\mathcal{D}'(M,B)$ be continuous operators.
We write 
\begin{equation} \label{e-gue160507f}
\mbox{$A\equiv \hat A$ (on $M\times M_1$)} 
\end{equation}
if $A-\hat A$ is a smoothing operator. If $M=M_1$, we simply write ``on $M$". 
We say that $A$ is properly supported if the restrictions of the two projections 
$(x,y)\mapsto x$, $(x,y)\mapsto y$ to ${\rm supp\,}(A(x,y))$
are proper.

Let $H(x,y)\in\mathcal D'(M\times M_1,B\boxtimes E^*)$. We write $H$ to denote the unique 
continuous operator $\mathcal C^\infty_0(M_1,E)\To\mathcal D'(M,B)$ with distribution kernel $H(x,y)$. 
In this work, we identify $H$ with $H(x,y)$.

\subsection{CR manifolds} 

Let $(X, T^{1,0}X)$ be a compact, connected and orientable CR manifold of dimension $2n+1$, $n\geq 1$, where $T^{1,0}X$ is a CR structure of $X$, that is, $T^{1,0}X$ is a subbundle of rank $n$ of the complexified tangent bundle $\mathbb{C}TX$, satisfying $T^{1,0}X\cap T^{0,1}X=\{0\}$, where $T^{0,1}X=\overline{T^{1,0}X}$, and $[\mathcal V,\mathcal V]\subset\mathcal V$, where $\mathcal V=\mathcal{C}^\infty(X, T^{1,0}X)$. There is a unique subbundle $HX$ of $TX$ such that $\mathbb{C}HX=T^{1,0}X \oplus T^{0,1}X$, i.e. $HX$ is the real part of $T^{1,0}X \oplus T^{0,1}X$. Let $J:HX\To HX$ be the complex structure map given by $J(u+\ol u)=\imath u-\imath\ol u$, for every $u\in T^{1,0}X$. 
By complex linear extension of $J$ to $\mathbb{C}TX$, the $\imath$-eigenspace of $J$ is $T^{1,0}X \, = \, \left\{ V \in \mathbb{C}HX \, ;\, JV \, =  \,  \imath V  \right\}.$ We shall also write $(X, HX, J)$ to denote a CR manifold.

We  fix a real non-vanishing $1$ form $\omega_0\in\mathcal{C}^{\infty}(X,T^*X)$ so that $\langle\,\omega_0(x)\,,\,u\,\rangle=0$, for every $u\in H_xX$, for every $x\in X$. 
For each $x \in X$, we define a quadratic form on $HX$ by
\[
\mathcal{L}_x(U,V) =\frac{1}{2}d\omega_0(JU, V), \forall \ U, V \in H_xX.
\]
We extend $\mathcal{L}$ to $\mathbb{C}HX$ by complex linear extension. Then, for $U, V \in T^{1,0}_xX$,
\begin{equation}\label{e-gue210813yyd}
\mathcal{L}_x(U,\overline{V}) = \frac{1}{2}d\omega_0(JU, \overline{V}) = -\frac{1}{2i}d\omega_0(U,\overline{V}).
\end{equation}
The Hermitian quadratic form $\mathcal{L}_x$ on $T^{1,0}_xX$ is called Levi form at $x$. In this paper, we always assume that the Levi form $\mathcal{L}$ on $T^{1,0}X$ is non-degenerate of constant signature $(n_-,n_+)$ on $X$, where $n_-$ denotes the number of negative eigenvalues of the Levi form and $n_+$ denotes the number of positive eigenvalues of the Levi form. Let $R\in\mathcal{C}^\infty(X,TX)$ be the non-vanishing vector field determined by 
\begin{equation}\label{e-gue170111ry}\begin{split}
\omega_0(R)=-1,\quad
d\omega_0(R,\cdot)\equiv0\ \ \mbox{on $TX$}.
\end{split}\end{equation}
Note that $X$ is a contact manifold with contact form $\omega_0$, contact plane $HX$ and $R$ is the Reeb vector field.

Fix a smooth Hermitian metric $\langle\, \cdot \,|\, \cdot \,\rangle$ on $\mathbb{C}TX$ so that $T^{1,0}X$ is orthogonal to $T^{0,1}X$, $\langle\, u \,|\, v \,\rangle$ is real if $u, v$ are real tangent vectors, $\langle\,R\,|\,R\,\rangle=1$ and $R$ is orthogonal to $T^{1,0}X\oplus T^{0,1}X$. For $u \in \mathbb{C}TX$, we write $|u|^2 := \langle\, u\, |\, u\, \rangle$. Denote by $T^{*1,0}X$ and $T^{*0,1}X$ the dual bundles of $T^{1,0}X$ and $T^{0,1}X$, respectively. They can be identified with subbundles of the complexified cotangent bundle $\mathbb{C}T^*X$. Define the vector bundle of $(0,q)$-forms by $T^{*0,q}X := \wedge^qT^{*0,1}X$. The Hermitian metric $\langle\, \cdot\, |\, \cdot \,\rangle$ on $\mathbb{C}TX$ induces, by duality, a Hermitian metric on $\mathbb{C}T^*X$ and also on the bundles of $(0,q)$ forms $T^{*0,q}X, q=0, 1, \cdots, n$. We shall also denote all these induced metrics by $\langle \,\cdot \,|\, \cdot \,\rangle$. For $u\in T^{*0,q}X$, we write $\abs{u}^2:=\langle\,u\,|\,u\,\rangle$. Note that we have the pointwise orthogonal decompositions:
\[
\renewcommand{\arraystretch}{1.2}
\begin{array}{c}
\mathbb{C}T^*X = T^{*1,0}X \oplus T^{*0,1}X \oplus \left\{ \lambda \omega_0;\, \lambda \in \mathbb{C} \right\}, \\
\mathbb{C}TX = T^{1,0}X \oplus T^{0,1}X \oplus \left\{ \lambda R;\, \lambda \in \mathbb{C} \right\}.
\end{array}
\]

Let $D$ be an open set of $X$. Let $\Omega^{0,q}(D)$ denote the space of smooth sections of $T^{*0,q}X$ over $D$ and let $\Omega^{0,q}_c(D)$ be the subspace of $\Omega^{0,q}(D)$ whose elements have compact support in $D$. Let 
\[
\ddbar_b:\Omega^{0,q}(X)\To\Omega^{0,q+1}(X)
\]
be the tangential Cauchy-Riemann operator. Let $dv(x)$ be the volume form on $X$ induced by the Hermitian metric $\langle\,\cdot\,|\,\cdot\,\rangle$.
The natural global $L^2$ inner product $(\,\cdot\,|\,\cdot\,)$ on $\Omega^{0,q}(X)$ 
induced by $dv(x)$ and $\langle\,\cdot\,|\,\cdot\,\rangle$ is given by
\[
(\,u\,|\,v\,):=\int_X\langle\,u(x)\,|\,v(x)\,\rangle\, dv(x)\,,\quad u,v\in\Omega^{0,q}(X)\,.
\]
We denote by $L^2_{(0,q)}(X)$ 
the completion of $\Omega^{0,q}(X)$ with respect to $(\,\cdot\,|\,\cdot\,)$. 
We extend $(\,\cdot\,|\,\cdot\,)$ to $L^2_{(0,q)}(X)$ 
in the standard way. For $f\in L^2_{(0,q)}(X)$, we denote $\norm{f}^2:=(\,f\,|\,f\,)$.
We extend
$\ddbar_{b}$ to $L^2_{(0,r)}(X)$, $r=0,1,\ldots,n$, by
\[
\ddbar_{b}:{\rm Dom\,}\ddbar_{b}\subset L^2_{(0,r)}(X)\To L^2_{(0,r+1)}(X)\,,
\]
where ${\rm Dom\,}\ddbar_{b}:=\{u\in L^2_{(0,r)}(X);\, \ddbar_{b}u\in L^2_{(0,r+1)}(X)\}$ and, for any $u\in L^2_{(0,r)}(X)$, $\ddbar_{b} u$ is defined in the sense of distributions.
We also write
\[
\ol{\pr}^{*}_{b}:{\rm Dom\,}\ol{\pr}^{*}_{b}\subset L^2_{(0,r+1)}(X)\To L^2_{(0,r)}(X)
\]
to denote the Hilbert space adjoint of $\ddbar_{b}$ in the $L^2$ space with respect to $(\,\cdot\,|\,\cdot\, )$.
Let $\Box^{q}_{b}$ denote the (Gaffney extension) of the Kohn Laplacian given by
\begin{equation}\label{e-suIX}
\renewcommand{\arraystretch}{1.2}
\begin{array}{c}
{\rm Dom\,}\Box^q_{b}=\Big\{s\in L^2_{(0,q)}(X);\, 
s\in{\rm Dom\,}\ddbar_{b}\cap{\rm Dom\,}\ol{\pr}^{*}_{b},\,
\ddbar_{b}s\in{\rm Dom\,}\ol{\pr}^{*}_{b}, \, \ol{\pr}^{*}_{b}s\in{\rm Dom\,}\ddbar_{b}\Big\}\,,\\
\Box^{q}_{b}s=\ddbar_{b}\ol{\pr}^{*}_{b}s+\ol{\pr}^{*}_{b}\ddbar_{b}s
\:\:\text{for $s\in {\rm Dom\,}\Box^q_{b}$}\,.
 \end{array}
\end{equation}
Let
\begin{equation}\label{e-suXI-I}
S^{(q)}:L^2_{(0,q)}(X)\To{\rm Ker\,}\Box^q_b
\end{equation}
be the orthogonal projection with respect to the $L^2$ inner product $(\,\cdot\,|\,\cdot\,)$ and let
\[
S^{(q)}(x,y)\in\mathcal{D}'(X\times X,T^{*0,q}X\boxtimes(T^{*0,q}X)^*)
\]
denote the distribution kernel of $S^{(q)}$. 

We recall H\"ormander symbol space. Let $D\subset X$ be a local coordinate patch with local coordinates $x=(x_1,\ldots,x_{2n+1})$. 

\begin{defn}\label{d-gue140221a}
For $m\in\Real$, $S^m_{1,0}(D\times D\times\mathbb{R}_+,T^{*0,q}X\boxtimes(T^{*0,q}X)^*)$ 
is the space of all $a(x,y,t)\in\mathcal{C}^\infty(D\times D\times\mathbb{R}_+,T^{*0,q}X\boxtimes(T^{*0,q}X)^*)$ 
such that, for all compact $K\Subset D\times D$ and all $\alpha, \beta\in\mathbb N^{2n+1}_0$, $\gamma\in\mathbb N_0$, 
there is a constant $C_{\alpha,\beta,\gamma}>0$ such that 
\[\abs{\pr^\alpha_x\pr^\beta_y\pr^\gamma_t a(x,y,t)}\leq C_{\alpha,\beta,\gamma}(1+\abs{t})^{m-\gamma},\ \ 
\mbox{for every $(x,y,t)\in K\times\Real_+, t\geq1$}.\]
Put 
\[
S^{-\infty}(D\times D\times\mathbb{R}_+,T^{*0,q}X\boxtimes(T^{*0,q}X)^*) :=\bigcap_{m\in\Real}S^m_{1,0}(D\times D\times\mathbb{R}_+,T^{*0,q}X\boxtimes(T^{*0,q}X)^*).
\]
Let $a_j\in S^{m_j}_{1,0}(D\times D\times\mathbb{R}_+,T^{*0,q}X\boxtimes(T^{*0,q}X)^*)$, 
$j=0,1,2,\ldots$ with $m_j\To-\infty$, as $j\To\infty$. 
Then there exists $a\in S^{m_0}_{1,0}(D\times D\times\mathbb{R}_+,T^{*0,q}X\boxtimes(T^{*0,q}X)^*)$ 
unique modulo $S^{-\infty}$, such that 
$a-\sum^{k-1}_{j=0}a_j\in S^{m_k}_{1,0}(D\times D\times\mathbb{R}_+,T^{*0,q}X\boxtimes(T^{*0,q}X)^*\big)$ 
for $k=0,1,2,\ldots$. 

If $a$ and $a_j$ have the properties above, we write $a\sim\sum^{\infty}_{j=0}a_j$ (in 
$S^{m_0}_{1,0}\big(D\times D\times\mathbb{R}_+,T^{*0,q}X\boxtimes(T^{*0,q}X)^*\big)$ ). 
We write
\[
s(x, y, t)\in S^{m}_{{\rm cl\,}}(D\times D\times\mathbb{R}_+,T^{*0,q}X\boxtimes(T^{*0,q}X)^*)
\]
if $s(x, y, t)\in S^{m}_{1,0}(D\times D\times\mathbb{R}_+,T^{*0,q}X\boxtimes(T^{*0,q}X)^*)$ and 
\[
\renewcommand{\arraystretch}{1.2}
\begin{array}{lc}
&s(x, y, t)\sim\sum^\infty_{j=0}s^j(x, y)t^{m-j}\text{ in }S^{m}_{1, 0}
(D\times D\times\mathbb{R}_+\,,T^{*0,q}X\boxtimes(T^{*0,q}X)^*)\,,\\
&s^j(x, y)\in\mathcal{C}^\infty(D\times D,T^{*0,q}X\boxtimes(T^{*0,q}X)^*),\ j\in\N_0.\end{array}
\]

We sometimes simply write $S^{m}_{1,0}$ to denote $S^m_{1,0}(D\times D\times\mathbb{R}_+,T^{*0,q}X\boxtimes(T^{*0,q}X)^*)$, $m\in\mathbb R\bigcup\set{-\infty}$.  
\end{defn} 

Let $E$ be a smooth
vector bundle over $X$. Let $m\in\Real$, $0\leq\rho,\delta\leq1$. Let $S^m_{\rho,\delta}(X,E)$
denote the H\"{o}rmander symbol space on $X$ with values in $E$ of order $m$ type $(\rho,\delta)$
and let $S^m_{{\rm cl\,}}(X,E)$
denote the space of classical symbols on $X$ with values in $E$ of order $m$. 
 For $a\in S^m_{{\rm cl\,}}(X,E)$, we write ${\rm Op\,}(a)$ to denote the pseudodifferential operator on $X$ of order $m$ from sections of $E$ to sections of $E$ with full symbol $a$. 

Let $D\subset X$ be an open set. Let
$L^m_{{\rm cl\,}}(D,T^{*0,q}X\boxtimes (T^{*0,q}X)^*)$
denote the space of classical
pseudodifferential operators on $D$ of order $m$ from sections of
$T^{*0,q}X$ to sections of $T^{*0,q}X$ respectively. Let $P\in L^m_{{\rm cl\,}}(D,T^{*0,q}X\boxtimes (T^{*0,q}X)^*)$. We write $\sigma_P$ and $\sigma^0_P$ to denote the full symbol of $P$ and the principal symbol of $P$ respectively. 

\section{Toeplitz operators on CR manifolds}\label{s-gue210628yyd}

In this section, we will study the Toeplitz operator $T_P:=S^{(q)}\circ P\circ S^{(q)}$, for pseudodifferential operator $P$ and we will establish star product for the Toeplitz operators. We will assume that $X$ is compact. We first recall some results about Szeg\H{o} kernel expansions for $(0,q)$ forms obtained in~\cite{hsiao}. We recall that in this paper, we always assume that the Levi form $\mathcal{L}$ on $T^{1,0}X$ is non-degenerate of constant signature $(n_-,n_+)$ on $X$ and We will always let $q=n_-$. 

\subsection{Szeg\H{o} kernel asymptotics}\label{s-gue210628yydI} 
The characteristic set of $\Box^q_b$ is given by 
\begin{equation}\label{e-gue210707yydI}
\begin{split}
&\Sigma=\Sigma^-\bigcup\Sigma^+,\\
&\Sigma^-=\set{(x,\lambda\omega_0(x))\in T^*X;\,\lambda<0},\\
&\Sigma^+=\set{(x,\lambda\omega_0(x))\in T^*X;\,\lambda>0}.
\end{split}
\end{equation}

We have the following (see~\cite[Theorem 1.2]{hsiao},~\cite[Theorem 4.7]{HM14})

\begin{thm}\label{t-gue161109I}
Suppose that $\Box^q_b$ has $L^2$ closed range. Then, there exist continuous operators
$S_-, S_+: L^2_{(0,q)}(X)\To{\rm Ker\,}\Box^q_b$ such that 
\begin{equation}\label{e-gue210707yyd}
\begin{split}
&S^{(q)}=S_-+S_+,\\
&{\rm WF'\,}(S_-)={\rm diag\,}(\Sigma^-\times\Sigma^-),\\
&S_+\equiv0\ \ \mbox{if $q\neq n_+$},\\
&{\rm WF'\,}(S_+)={\rm diag\,}(\Sigma^+\times\Sigma^+)\ \ \mbox{if $q=n_-=n_+$},
\end{split}
\end{equation}
where ${\rm WF'\,}(S_-)=\set{(x,\xi,y,\eta)\in T^*X\times T^*X;\,(x,\xi,y,-\eta)\in{\rm WF\,}(S_-)}$, ${\rm WF\,}(S_-)$ is the wave front set of $S_-$ in the sense of H\"ormander. 

Moreover, let $D\subset X$ be any local coordinate patch with local coordinates $x=(x_1,\ldots,x_{2n+1})$, then 
 $S_-(x,y)$, $S_+(x,y)$ satisfy
\[
S_{\mp}(x, y)\equiv\int^{\infty}_{0}e^{i\varphi_{\mp}(x, y)t}s_{\mp}(x, y, t)dt\ \ \mbox{on $D$},
\]
with 
\begin{equation}  \label{e-gue161110r}
\renewcommand{\arraystretch}{1.2}
\begin{array}{ll}
&s_-(x, y, t), s_+(x,y,t)\in S^{n}_{{\rm cl\,}}(D\times D\times\mathbb{R}_+,T^{*0,q}X\boxtimes(T^{*0,q}X)^*), \\
&s_{\mp}(x,y,t)\sim\sum^\infty_{j=0}s^j_{\mp}(x, y)t^{n-j}\text{ in }S^{n}_{1, 0}
(D\times D\times\mathbb{R}_+\,,T^{*0,q}X\boxtimes(T^{*0,q}X)^*),\\
&s^j_{\mp}(x, y)\in\mathcal{C}^\infty(D\times D,T^{*0,q}X\boxtimes(T^{*0,q}X)^*),\ j\in\mathbb N_0,\\
&s_+(x,y,t)=0\ \ \mbox{if $q\neq n_+$},\\
&s^0_-(x,x)\neq0,\ \ \mbox{for all $x\in D$},
\end{array}
\end{equation}
and the phase functions $\varphi_-$, $\varphi_+$  satisfy
\[
\renewcommand{\arraystretch}{1.2}
\begin{array}{ll}
&\varphi_+(x,y), \varphi_-\in\mathcal{C}^\infty(D\times D),\ \ {\rm Im\,}\varphi_{\mp}(x, y)\geq0,\\
&\varphi_-(x, x)=0,\ \ \varphi_-(x, y)\neq0\ \ \mbox{if}\ \ x\neq y,\\
&d_x\varphi_-(x, y)\big|_{x=y}=-\omega_0(x), \ \ d_y\varphi_-(x, y)\big|_{x=y}=\omega_0(x), \\
&-\ol\varphi_+(x, y)=\varphi_-(x,y).
\end{array}
\]
\end{thm}

\begin{rem}\label{r-gue180404}
 Kohn~\cite{Koh86} proved that if $q=n_-=n_+$ or $\abs{n_--n_+}>1$ then $\Box^q_b$ has $L^2$ closed range.
\end{rem} 

The following result describes the phase function in local coordinates (see chapter 8 of part I in \cite{hsiao})

\begin{thm} \label{t-gue161110g}
For a given point $p\in X$, let $\{W_j\}_{j=1}^{n}$
be an orthonormal frame of $T^{1, 0}X$ in a neighborhood of $p$
such that
the Levi form is diagonal at $p$, i.e.\ $\mathcal{L}_{p}(W_{j},\overline{W}_{s})=\delta_{j,s}\mu_{j}$, $j,s=1,\ldots,n$.
We take local coordinates
$x=(x_1,\ldots,x_{2n+1})$, $z_j=x_{2j-1}+ix_{2j}$, $j=1,2,\ldots,n$,
defined on some neighborhood of $p$ such that $\omega_0(p)=dx_{2n+1}$, $x(p)=0$, 
and, for some $c_j\in\Complex$, $j=1,\ldots,n$\,,
\begin{equation}\label{e-gue161219a}
\begin{split}
&R=-\frac{\pr}{\pr x_{2n+1}},\\
&W_j=\frac{\pr}{\pr z_j}-i\mu_j\ol z_j\frac{\pr}{\pr x_{2n+1}}-
c_jx_{2n+1}\frac{\pr}{\pr x_{2n+1}}+\sum^{2n}_{k=1}a_{j,k}(x)\frac{\pr}{\pr x_k}+O(\abs{x}^2),\ j=1,\ldots,n\,,
\end{split}
\end{equation}
where $a_{j,k}(x)\in\mathcal{C}^\infty$, $a_{j,k}(x)=O(\abs{x})$, for every $j=1,\ldots,n$, $k=1,\ldots,2n$. 
Set
$y=(y_1,\ldots,y_{2n+1})$, $w_j=y_{2j-1}+iy_{2j}$, $j=1,2,\ldots,n$.
Then, for $\varphi_-$ in Theorem~\ref{t-gue161109I}, we have
\begin{equation} \label{e-gue140205VI}
{\rm Im\,}\varphi_-(x,y)\geq c\sum^{2n}_{j=1}\abs{x_j-y_j}^2,\ \ c>0,
\end{equation}
in some neighbourhood of $(0,0)$ and
\begin{equation} \label{e-gue140205VII}
\renewcommand{\arraystretch}{1.2}
\begin{array}{l}
\varphi_-(x, y)=-x_{2n+1}+y_{2n+1}+i\sum^{n}_{j=1}\abs{\mu_j}\abs{z_j-w_j}^2 +\sum^{n}_{j=1}\Bigr(i\mu_j(\ol z_jw_j-z_j\ol w_j)\\
\ +c_j(-z_jx_{2n+1}+w_jy_{2n+1})+\ol c_j(-\ol z_jx_{2n+1}+\ol w_jy_{2n+1})\Bigr)\\
\ +(x_{2n+1}-y_{2n+1})f(x, y) +O(\abs{(x, y)}^3),
\end{array}
\end{equation}
where $f$ is smooth and satisfies $f(0,0)=0$.
\end{thm} 

We pause and introduce some notations. For a given point $p\in X$, let $\{W_j\}_{j=1}^{n}$ be an
orthonormal frame of $(T^{1,0}X,\langle\,\cdot\,|\,\cdot\,\rangle)$ near $p$, for which the Levi form
is diagonal at $p$. Put
\[
\mathcal{L}_{p}(W_j,\ol W_\ell)=\mu_j(p)\delta_{j,\ell}\,,\;\; j,\ell=1,\ldots,n\,.
\]
We will denote by
\begin{equation}\label{det140530}
\det\mathcal{L}_{p}=\prod_{j=1}^{n}\mu_j(p)\,.
\end{equation}
Let $\{e_j\}_{j=1}^{n}$ denote the basis of $T^{*0,1}X$, dual to $\{\ol W_j\}^{n}_{j=1}$. We assume that
$\mu_j(p)<0$ if $1\leq j\leq n_-$ and $\mu_j(p)>0$ if $n_-+1\leq j\leq n$. Put
\[
\renewcommand{\arraystretch}{1.2}
\begin{array}{ll}
&\mathcal{N}(p,n_-):=\set{ce_1(p)\wedge\ldots\wedge e_{n_-}(p);\, c\in\Complex},\\
&\mathcal{N}(p,n_+):=\set{ce_{n_-+1}(p)\wedge\ldots\wedge e_{n}(p);\, c\in\Complex}
\end{array}
\]
and let
\begin{equation}\label{tau140530}
\tau_{p,n_-}:T^{*0,q}_{p}X\To\mathcal{N}(p,n_-)\,,\quad
\tau_{p,n_+}:T^{*0,q}_{p}X\To\mathcal{N}(p,n_+)\,,
\end{equation}
be the orthogonal projections onto $\mathcal{N}(p,n_-)$ and $\mathcal{N}(p,n_+)$
with respect to $\langle\,\cdot\,|\,\cdot\,\rangle$, respectively. For $J=(j_1,\ldots,j_q)$, $1\leq j_1<\cdots<j_q\leq n$, let 
$e_J:=e_{j_1}\wedge\cdots\wedge e_{j_q}$. For $\abs{I}=\abs{J}=q$, $I$, $J$ are strictly increasing, let $e_I\otimes(e_J)^*$ be the linear transformation from 
$T^{*0,q}X$ to $T^{*0,q}X$ given by 
\[(e_I\otimes(e_J)^*)(e_K)=\delta_{J,K}e_I,\]
for every $\abs{K}=q$, $K$ is strictly increasing, where $\delta_{J,K}=1$ if $J=K$, $\delta_{J,K}=0$ if $J\neq K$. For any $f\in T^{*0,q}X\boxtimes(T^{*0,q}X)^*$, we have \[f=\sideset{}{'}\sum_{\abs{I}=\abs{J}=q}c_{I,J}e_I\otimes(e_J)^*,\]
$c_{I,J}\in\mathbb C$, for all $\abs{I}=\abs{J}=q$, $I$, $J$ are strictly increasing, where $\sum'$ means that the summation is performed only over strictly increasing multi-indices.  We call $c_{I,J}e_I\otimes(e_J)^*$ the component of $f$ in the direction $e_I\otimes(e_J)^*$. Let $I_0=(1,2,\ldots,q)$. We can check that 
\[\tau_{p,n_-}=e_{I_0}(p)\otimes(e_{I_0}(p))^*.\]

The following formula for the leading term $s^0_-$ on the diagonal follows from \cite[\S 9]{hsiao}. The formula for the leading term $s^0_+$ on the diagonal follows similarly.

\begin{thm} \label{t-gue140205III}
For the leading term $s^0_{\mp}(x,y)$ of the expansion \eqref{e-gue161110r} of $s_{\mp}(x,y,t)$, we have
\[
s^0_{\mp}(x_0, x_0)=\frac{1}{2}\pi^{-n-1}\abs{\det\mathcal{L}_{x_0}}\tau_{x_0,n_{\mp}}\,,\:\:x_0\in D.
\]
\end{thm} 

\subsection{Toeplitz operators on CR manifolds}\label{s-gue210628yydII}

From now on, we assume that $\Box^q_b$ has $L^2$ closed range. We need 

\begin{defn}\label{d-gue210624yyd}
Let $H: \Omega^{0,q}(X)\To\Omega^{0,q}(X)$ be a continuous operator with distribution kernel $H(x,y)\in\mathcal{D}'(X\times X,T^{*0,q}X\boxtimes(T^{*0,q}X)^*)$. 
We say that $H$ is a complex Fourier integral operator of Szeg\H{o} type of order $k\in\mathbb Z$ if $H$ is smoothing away the diagonal on $X$ and 
for every local coordinate patch $D\subset X$ with local coordinates $x=(x_1,\ldots,x_{2n+1})$, we have 
\[\begin{split}
&H(x,y)\equiv H_-(x,y)+H_+(x,y)\ \ \mbox{on $D$},\\
&\mbox{$H_-(x,y)\equiv\int^\infty_0e^{i\varphi_{-}(x, y)t}a_-(x, y, t)dt$ on $D$},\\
&\mbox{$H_+(x,y)\equiv\int^\infty_0e^{i\varphi_{+}(x, y)t}a_+(x, y, t)dt$ on $D$},\end{split}\]
where $a_-, a_+\in S^{k}_{{\rm cl\,}}(D\times D\times\mathbb{R}_+,T^{*0,q}X\boxtimes(T^{*0,q}X)^*)$, $a_+=0$ if $q\neq n_+$, $\varphi_-$, $\varphi_+$ are as in 
Theorem~\ref{t-gue161110g}. We write $\sigma^0_{H,-}(x,y)$ to denote the leading term of the expansion \eqref{e-gue161110r} of $a_{-}(x,y,t)$. If $q=n_+$, we write $\sigma^0_{H,+}(x,y)$ to denote the leading term of the expansion \eqref{e-gue161110r} of $a_{+}(x,y,t)$. Note that $\sigma^0_{H,-}(x,y)$ and $\sigma^0_{H,+}(x,y)$ depend on the choices of the phases $\varphi_-$ and $\varphi_+$ but $\sigma^0_{H,-}(x,x)$ and $\sigma^0_{H,+}(x,x)$ are independent of the choices of the phases $\varphi_-$ and $\varphi_+$.

Let $\Psi_k(X)$ denote the space of all complex Fourier integral operators of Szeg\H{o} type of order $k$. 
\end{defn}

We need 

\begin{lem}\label{l-gue210624yyd}
Let $B\in\Psi_k(X)$ with $S^{(q)}B\equiv B\equiv BS^{(q)}$. Let $x=(x_1,\ldots,x_{2n+1})$ be local coordinates of $X$ defined on an open set $D\subset X$. 
Assume that 
\[\mbox{$\tau_{x,n_-}\sigma^0_{B,-}(x,x)\tau_{x,n_-}=0$, for every $x\in D$},\]
and 
\[\mbox{$\tau_{x,n_+}\sigma^0_{B,+}(x,x)\tau_{x,n_+}=0$, for every $x\in D$},\]
if $q=n_+$. Then, $B\in\Psi_{k-1}(X)$. 
\end{lem}

\begin{proof}
Let $B_1:=B+B^*$, where $B^*$ is the adjoint of $B$ with respect to $(\,\cdot\,|\,\cdot\,)$. We first assume that $q\neq n_+$. 
Let $x=(x_1,\ldots,x_{2n+1})$ be local coordinates of $X$ defined on an open set $D\subset X$. Fix $p\in D$. Take $x=(x_1,\ldots,x_{2n+1})$ so that $x(p)=0$ and \eqref{e-gue161219a} hold. We write 
\[\begin{split}
&\mbox{$B_1(x,y)\equiv\int^\infty_0e^{i\varphi_{-}(x, y)t}b(x, y, t)dt$ on $D$},\\
&b(x, y, t)\sim\sum^\infty_{j=0}b_j(x, y)t^{k-j}\text{ in }S^{k}_{1, 0}
(D\times D\times\mathbb{R}_+\,,T^{*0,q}X\boxtimes(T^{*0,q}X)^*),\\
&b_j(x, y)\in\mathcal{C}^\infty(D\times D,T^{*0,q}X\boxtimes(T^{*0,q}X)^*),\ j\in\mathbb N_0.
\end{split}\]
From Malgrange preparation theorem, we assume that 
\begin{equation}\label{e-gue210701yyd}
\varphi_-(x,y)=-x_{2n+1}+y_{2n+1}+\hat\varphi_-(x,y'),\ \ \hat\varphi_-(x,y')\in\mathcal{C}^\infty(D\times D),
\end{equation}
where $y'=(y_1,\ldots,y_{2n})$. 
By using integration by parts with respect to $t$, we assume that 
$s^j_-(x,y)$, $b_j(x,y)$, $s_-(x,y,t)$, $b(x,y,t)$ are independent of $y_{2n+1}$, for every $j=0,1,\ldots$, where $s_-(x,y,t)$, $s^j_-(x,y)$, $j=0,1,\ldots$, are as in \eqref{e-gue161110r}. Since $B_1$ and $S^{(q)}$ are smoothing away the diagonal, we assume that  the projections 
\[\begin{split}
&(x,y')\in{\rm supp\,}b(x,y,t)\To x,\ \ (x,y')\in{\rm supp\,}b_j(x,y)\To x,\\
&(x,y')\in{\rm supp\,}s_-(x,y,t)\To x,\ \ (x,y')\in{\rm supp\,}s^j_-(x,y)\To x,
\end{split} \]
are proper on $D$, for every $j=0,1,\ldots$. Fix $D_0\Subset D$, $D_0$ is an open set of 
$0$. Let $\chi\in\mathcal{C}^{\infty}_c(\mathbb R,[0,1])$, $\chi\equiv1$ on $[-\frac{1}{2},\frac{1}{2}]$. Let $\varepsilon>0$ be a small constant so that $\frac{x_{2n+1}-y_{2n+1}}{\varepsilon}\notin{\rm supp\,}\chi$, for every $(x',x_{2n+1})\in D_0$ and every $(y',y_{2n+1})\notin D$.
Let $g\in\Omega^{0,q}_c(D_0)$. From $B_1=B_1S^{(q)}$, we have on $D_0$, 
\[\begin{split}
&(B_1g)(x)\\
&=\int_X\int^\infty_0\int_X\int^\infty_0e^{i\varphi_{-}(x, u)\gamma+i\varphi_-(u,y)t}\chi(\frac{x_{2n+1}-u_{2n+1}}{\varepsilon})b(x,u',\gamma)s_-(u,y',t)g(y)dv(y)dtdv(u)d\gamma \\
& \quad +(Fg)(x)
\\
&= \int_X\int^\infty_0\int_X\int^\infty_0e^{it(\varphi_{-}(x, u)\sigma+\varphi_-(u,y))}\chi(\frac{x_{2n+1}-u_{2n+1}}{\varepsilon})tb(x,u',t\sigma)s_-(u,y',t)g(y)dv(y)d\sigma dv(u)dt \\ 
& \quad +(Fg)(x)\,,
\end{split}\]
where $F\equiv0$ on $D_0$. Hence, 
\begin{equation}\label{e-gue210626yyd}
B_1(x,y)\equiv \int^\infty_0\int_X e^{it(\varphi_{-}(x, u)\sigma+\varphi_-(u,y))}\chi(\frac{1-\sigma}{\varepsilon})\chi(\frac{x_{2n+1}-u_{2n+1}}{\varepsilon})tb(x,u',t\sigma)s_-(u,y',t)d\sigma dv(u).
 \end{equation}
 For every $f\in\mathcal{C}^\infty(D\times D)$ ($h\in\mathcal{C}^\infty(D)$), we write $\Td f\in\mathcal{C}^\infty(D^{\mathbb C}\times D^{\mathbb C})$ ($h\in\mathcal{C}^\infty(D^{\mathbb C})$) to denote an almost analytic extension of $f$ ($h$), where $D^{\mathbb C}$ is an open set of $\mathbb C^{2n+1}$ with 
 $D^{\mathbb C}\bigcap\mathbb R^{2n+1}=D$. We take $\Td\varphi_-$ so that 
 \begin{equation}\label{e-gue210626yydI}
 \Td\varphi_-(\Td x,\Td y)=-\Td x_{2n+1}+\Td y_{2n+1}+\Td{\hat\varphi_-}(\Td x,\Td y').
  \end{equation}
  Let $\beta(\Td x,\Td y)=(\beta_1(\Td x,\Td y),\ldots,\beta_{2n+1}(\Td x,\Td y))\in\mathcal{C}^\infty(D^{\mathbb C}\times D^{\mathbb C},\mathbb C^{2n+1})$, 
  $\gamma(\Td x,\Td y)\in\mathcal{C}^\infty(D^{\mathbb C}\times D^{\mathbb C},\mathbb C)$ be the solution of the system 
  \begin{equation}\label{e-gue210701yydII}
\begin{cases}
\frac{\pr\Td\varphi_-}{\pr\Td y_j}(\Td x,\beta(\Td x,\Td y))\gamma(\Td x,\Td y)+\frac{\pr\Td\varphi_-}{\pr\Td x_j}(\beta(\Td x,\Td y),\Td y)=0,\ \ j=1,\ldots,2n+1,\\
\Td\varphi_-(\Td x,\beta(\Td x,\Td y))=0. 
\end{cases}\end{equation}
From \eqref{e-gue210626yydI}, we see that $\beta(\Td x,\Td y)$ and $\gamma(\Td x,\Td y)$ are independent of $\Td y_{2n+1}$. From complex stationary phase formula of Melin-Sj\"ostrand~\cite{ms}, we get $B_1(x,y)\equiv \int^{+\infty}_0 e^{it\varphi_1(x,y)}f(x,y,t)dt$ on $D_0$, where 
\begin{equation}\label{e-gue210626yydII}\
\varphi_1(x,y)=\Td\varphi_-(\beta(x,y'),y)=-\beta_{2n+1}(x,y')+y_{2n+1}+\Td{\hat\varphi_-}(\beta(x,y'),y'),
\end{equation}
$f(x, y, t)\sim\sum^\infty_{j=0}f_j(x, y)t^{k-j}\text{ in }S^{k}_{1, 0}
(D_0\times D_0\times\mathbb{R}_+\,,T^{*0,q}X\boxtimes(T^{*0,q}X)^*)$, 
$f_j(x, y)\in\mathcal{C}^\infty(D_0\times D_0,T^{*0,q}X\boxtimes(T^{*0,q}X)^*)$, $j\in\mathbb N_0$, 
\begin{equation}\label{e-gue210626yydIII}
f_0(x,y)=f_0(x,y')=c(x,y')\Td b_0(x,\beta(x,y'))\Td s^0_-(\beta(x,y'),y'),
\end{equation}
where $c(x,y')\in\mathcal{C}^\infty(D_0\times D_0,\mathbb C)$, $c(x,y')\neq0$, for every $(x,y)\in D_0\times D_0$. From~\cite[Theorem 5.4]{HM14}, there is a $g(x,y)\in\mathcal{C}^\infty(D_0\times D_0)$, $g(x,y)\neq0$ at every $(x,y)\in D_0\times D_0$, such that $\varphi_-(x,y)-\varphi_1(x,y)-g(x,y)\varphi_-(x,y)=O(\abs{x-y}^N)$, for every $N\in\mathbb N$. From $\frac{\pr\varphi_-}{\pr y_{2n+1}}(x,y)=\frac{\pr\varphi_1}{\pr y_{2n+1}}(x,y)=1$, we deduce that $\varphi_-(x,y)-\varphi_1(x,y)=O(\abs{x-y}^N)$, for every $N\in\mathbb N$.  Hence, we can replace $\varphi_1$ by $\varphi_-$ and we have 
\begin{equation}\label{e-gue210626ycd}
\int^{+\infty}_0 e^{it\varphi_-(x,y)}b(x,y,t)dt\equiv \int^{+\infty}_0 e^{it\varphi_-(x,y)}f(x,y,t)dt\ \ \mbox{on $D_0$}.
\end{equation}
From \eqref{e-gue210626ycd}, we deduce that $b_0(x,y)=f_0(x,y)+h(x,y)\varphi_-(x,y)+O(\abs{x-y}^N)$, for every $N\in\mathbb N$, for some $h(x,y)\in\mathcal{C}^\infty(D_0\times D_0)$. Since $b_0(x,y)$ and $f_0(x,y)$ are independent of $y_{2n+1}$, we deduce that $b_0(x,y)=f_0(x,y)+O(\abs{x-y}^N)$, for every $N\in\mathbb N$. From 
this observation and \eqref{e-gue210626yydIII}, we get 
\begin{equation}\label{e-gue210627yyd}
\mbox{$b_0(x,y)=c(x,y)\Td b_0(x,\beta(x,y))\Td s^0_-(\beta(x,y),y)+O(\abs{x-y}^N)$, for every $N\in\mathbb N$}.
\end{equation}
From Theorem~\ref{t-gue140205III} and \eqref{e-gue210627yyd}, we deduce 
\begin{equation}\label{e-gue210627yydI}
b_0(x,y)(I-\tau_{y,n_-})=O(\abs{x-y}). 
\end{equation}
Form $B_1=S^{(q)}B_1$, we can repeat the procedure above and deduce that 
\begin{equation}\label{e-gue210627yydII}
(I-\tau_{x,n_-})b_0(x,y)=O(\abs{x-y}).
\end{equation}
From \eqref{e-gue210627yydI}, \eqref{e-gue210627yydII} and by the assumption that $\tau_{x,n_-}b_0(x,y)\tau_{y,n_-}=O(\abs{x-y})$, we conclude that 
\begin{equation}\label{e-gue210627yydIII}
b_0(x,y)=O(\abs{x-y}).
\end{equation} 

We assume that $b_0(x,y)=O(\abs{(x,y)}^{N_0})$, for some $N_0\in\mathbb N$. We are going to prove that $b_0(x,y)=O(\abs{(x,y)}^{N_0+1})$. Again, from 
\eqref{e-gue210627yyd}, we have 
\begin{equation}\label{e-gue210627yyda}
b_0(x,y)(I-\tau_{y,n_-})=O(\abs{(x,y)}^{N_0+1}) 
\end{equation}
and similarly, 
\begin{equation}\label{e-gue210627yydb}
(I-\tau_{x,n_-})b_0(x,y)=O(\abs{(x,y)}^{N_0+1}). 
\end{equation}
We only need to prove 
\begin{equation}\label{e-gue210627ycd}
\tau_{x,n_-}b_0(x,y)\tau_{y,n_-}=O(\abs{(x,y)}^{N_0+1}). 
\end{equation} 

Suppose that $\mu_j<0$, $j=1,\ldots,n_-$, $\mu_j>0$, $j=n_-+1,\ldots,n$. 
We will use the same notations as in the discussion after Theorem~\ref{t-gue161110g}.
Let $I_0=(1,\ldots,q)$. Write $b_0(x,y)=\sum'_{\abs{I}=\abs{J}=q}b_{0,I,J}(x,y)e_I(x)\otimes(e_J(y))^*$, $b_{0,I,J}(x,y)\in\mathcal{C}^\infty(D\times D)$, for every $\abs{I}=\abs{J}=q$, $I$, $J$ are strictly increasing. To prove \eqref{e-gue210627ycd}, we only need to prove 
\begin{equation}\label{e-gue210627ycdI}
b_{0,I_0,I_0}(x,y)=O(\abs{(x,y)}^{N_0+1}). 
\end{equation}
From $\ddbar_bB_1\equiv0$, we have $\ddbar_b\varphi_-=h_1\varphi_-+O(\abs{x-y}^N)$, for every $N\in\mathbb N$, for some smooth function $h_1(x,y)$. Since $\ddbar_b\varphi_-$ is independent of $y_{2n+1}$, we deduce that 
\begin{equation}\label{e-gue210627ycdII}
\mbox{$\ddbar_b\varphi_-=O(\abs{x-y}^N)$, for every $N\in\mathbb N$}.
\end{equation}
From \eqref{e-gue210627ycdII} and $\ddbar_bB_1\equiv0$, we get $(\ddbar_bb_0)(x,y)=h_2\varphi_-+O(\abs{x-y}^N)$, for every $N\in\mathbb N$, for some smooth function $h_2$. Since $\ddbar_bb_0$ is independent of $y_{2n+1}$, we get 
\begin{equation}\label{e-gue210627ycdIII}
\mbox{$(\ddbar_bb_0)(x,y)=O(\abs{x-y}^N)$, for every $N\in\mathbb N$}.
\end{equation}
From \eqref{e-gue210627yyda}, \eqref{e-gue210627yydb} and \eqref{e-gue210627ycdIII}, we deduce that 
\begin{equation}\label{e-gue210627ycda}
\frac{\pr}{\pr\ol z_j}b_{0,I_0,I_0}(x,y)=O(\abs{(x,y)}^{N_0}),\ \ j=q+1,\ldots,n.
\end{equation}
Similar, from $\ddbar^*_bB_1\equiv0$, we can repeat the procedure above with minor change and deduce that 
\begin{equation}\label{e-gue210627ycdb}
\frac{\pr}{\pr z_j}b_{0,I_0,I_0}(x,y)=O(\abs{(x,y)}^{N_0}),\ \ j=1,\ldots,q. 
\end{equation}

From $B^*_1=B_1$, it is straightforward to there are $h_3, h_4\in\mathcal{C}^\infty(D\times D)$ such that 
\begin{equation}\label{e-gue210627ycdc}
\mbox{$b^*_0(y,x)=h_3(x,y)b_0(x,y)+h_4(x,y)\varphi_-(x,y)+O(\abs{x-y}^N)$, for every $N\in\mathbb N$},
\end{equation}
where $b^*_0$ is the adjoint of $b_0$. From \eqref{e-gue210627ycdc} and notice that $b^*_0(y,x)=b^*_0(y,x')$, we get 
\[\begin{split}
&b^*_0(y,x)=\Td h_3((x',y_{2n+1}-\ol{\hat\varphi_-}(y,x')),y)\Td b_0((x',y_{2n+1}-\ol{\hat\varphi_-}(y,x')),y')\\
&\quad\mbox{$+O(\abs{x-y}^N)$, for every $N\in\mathbb N$},
\end{split}\]
where $x'=(x_1,\ldots,x_{2n})$. Hence 
\begin{equation}\label{e-gue210627yydh}
\begin{split}
&\ol{b_{0,I_0,I_0}}(y,x)=\Td h_3((x',y_{2n+1}-\ol{\hat\varphi_-}(y,x')),y)\Td b_{0,I_0,I_0}((x',y_{2n+1}-\ol{\hat\varphi_-}(y,x')),y')\\
&\quad\mbox{$+O(\abs{x-y}^N)$, for every $N\in\mathbb N$},
\end{split}
\end{equation}
From \eqref{e-gue210627ycda}, \eqref{e-gue210627ycdb} and \eqref{e-gue210627yydh}, we conclude that 
\begin{equation}\label{e-gue210627ycdp}
\frac{\pr}{\pr w_j}b_{0,I_0,I_0}(x,y)=O(\abs{(x,y)}^{N_0}),\ \ j=q+1,\ldots,n,
\end{equation}
\begin{equation}\label{e-gue210627ycdq}
\frac{\pr}{\pr\ol w_j}b_{0,I_0,I_0}(x,y)=O(\abs{(x,y)}^{N_0}),\ \ j=1,\ldots,q.
\end{equation}

From $\tau_{x,n_-}b_0(x,x)\tau_{x,n_-}=0$ and induction assumption, we can check that 
\begin{equation}\label{e-gue210627ycdr}
b_{0,I_0,I_0}(x,x)=O(\abs{x}^{N_0+1}). 
\end{equation}
 Fix $j\in\{q+1,\ldots,n\}$ and fix $\alpha, \beta\in\mathbb N_0$, $\alpha+\beta=N_0$. 
From\eqref{e-gue210627ycdr}, we have 
\begin{equation}\label{e-gue210630yyd}
\Bigr(\bigr((\frac{\pr}{\pr z_j}+\frac{\pr}{\pr  w_j})^\alpha(\frac{\pr}{\pr \ol z_j}+\frac{\pr}{\pr\ol w_j})^\beta\bigr)b_{0,I_0,I_0}\Bigr)(0,0)=0. 
\end{equation}
From \eqref{e-gue210630yyd}, we have 
\begin{equation}\label{e-gue201226yydh}
\begin{split}
&\Bigr((\frac{\pr}{\pr z_j})^{\alpha}(\frac{\pr}{\pr \ol w_j})^{\beta}b_{0,I_0,I_0}\Bigr)(0,0)=\sum_{\alpha_1, \alpha_2, \beta_1,\beta_2\in\mathbb N_0, \alpha_1+\alpha_2=\alpha, \beta_1+\beta_2=\beta, \alpha_2+\beta_1>0} c_{\alpha_1,\alpha_2,\beta_1,\beta_2}\\
&\times\Bigr((\frac{\pr}{\pr z_j})^{\alpha_1}(\frac{\pr}{\pr w_j})^{\alpha_2}(\frac{\pr}{\pr\ol z_j})^{\beta_1}(\frac{\pr}{\pr \ol w_j})^{\beta_2}b_{0,I_0,I_0}\Bigr)(0,0),
\end{split}
\end{equation}
where $c_{\alpha_1,\alpha_2,\beta_1,\beta_2}$ is a constant, for every $\alpha_1, \alpha_2, \beta_1,\beta_2\in\mathbb N_0$, $\alpha_1+\alpha_2=\alpha$, $\beta_1+\beta_2=\beta$, $\alpha_2+\beta_1>0$. Since $\alpha_2+\beta_1>0$, from \eqref{e-gue210627ycda} and \eqref{e-gue210627ycdp}, we get 
\[\begin{split}
&\Bigr((\frac{\pr}{\pr z_j})^{\alpha_1}(\frac{\pr}{\pr w_j})^{\alpha_2}(\frac{\pr}{\pr\ol z_j})^{\beta_1}(\frac{\pr}{\pr \ol w_j})^{\beta_2}b_{0,I_0,I_0}\Bigr)(0,0)=0,\\
&\mbox{for every 
$\alpha_1, \alpha_2, \beta_1,\beta_2\in\mathbb N_0$, $\alpha_1+\alpha_2=\alpha$, $\beta_1+\beta_2=\beta$, $\alpha_2+\beta_1>0$}.\end{split}\]
From this observation and \eqref{e-gue201226yydh}, we get 
\begin{equation}\label{e-gue201226yyde}
\mbox{$\Bigr((\frac{\pr}{\pr z_j})^{\alpha}(\frac{\pr}{\pr \ol w_j})^{\beta}b_{0,I_0,I_0}\Bigr)(0,0)=0$, for every $\alpha, \beta\in\mathbb N_0$, $\alpha+\beta=N_0$}. 
\end{equation}
Similarly, fix $j\in\{1,\ldots,q\}$, we can repeat the procedure above and deduce that 
\begin{equation}\label{e-gue210630yyda}
\mbox{$\Bigr((\frac{\pr}{\pr\ol z_j})^{\alpha}(\frac{\pr}{\pr w_j})^{\beta}b_{0,I_0,I_0}\Bigr)(0,0)=0$, for every $\alpha, \beta\in\mathbb N_0$, $\alpha+\beta=N_0$}. 
\end{equation}

Since $b_{0,I_0,I_0}$ is independent of $y_{2n+1}$ , from \eqref{e-gue210627ycdr}, we have 
\begin{equation}\label{e-gue201226ycdI}
\Bigr((\frac{\pr}{\pr x_{2n+1}})^Nb_{0,I_0,I_0}\Bigr)(0,0)=0,\ \ \mbox{for every $N\in\mathbb N$ with $\abs{N}\leq N_0$}. 
\end{equation}
From \eqref{e-gue201226ycdI}, we can repeat the proof of \eqref{e-gue201226yyde} with minor change and deduce that for $j\in\set{q+1,\ldots,n}$, $\ell\in\set{1,\ldots,q}$, we have 
\begin{equation}\label{e-gue201226ycdII}
\begin{split}
&\Bigr((\frac{\pr}{\pr z_j})^{\alpha_0}(\frac{\pr}{\pr \ol w_j})^{\beta_0}(\frac{\pr}{\pr\ol z_\ell})^{\alpha}\frac{\pr}{\pr w_\ell})^{\beta}(\frac{\pr}{\pr x_{2n+1}})^{\gamma}b_{0,I_0,I_0}\Bigr)(0,0)=0,\\
&\mbox{ for every $\alpha_0, \beta_0,\alpha, \beta\in\mathbb N^n_0$, $\gamma\in\mathbb N_0$, $|\alpha_0|+\abs{\beta_0}+\abs{\alpha}+|\beta|+|\gamma|=N_0$}.
\end{split}
\end{equation}

From \eqref{e-gue210627ycda}, \eqref{e-gue210627ycdb}, \eqref{e-gue210627ycdp}, \eqref{e-gue210627ycdq}, \eqref{e-gue201226yyde}, \eqref{e-gue210630yyda} and 
\eqref{e-gue201226ycdII}, we get \eqref{e-gue210627ycd}. By induction, $b_0(x,y)$ vanishes to infinite order at $(p,p)$. 

We have proved that $b_0$ vanishes to infinite order at $x=y$. Thus, we can take $b_0=0$ and $B_1\in\Psi_{k-1}(X)$. We can repeat the procedure above and get that $B_2:=iB-iB^*$ is in $\Psi_{k-1}(X)$. Hence, $B\in\Psi_{k-1}$. If $q=n_-=n_+$, the proof is similar. The lemma follows. 
\end{proof} 

\begin{defn}\label{d-gue210717yyd}
Let $P\in L^\ell_{{\rm cl\,}}(X,T^{*0,q}X\boxtimes(T^{*0,q}X)^*)$ with scalar principal symbol, $\ell\leq0$, $\ell\in\mathbb Z$.
The Toeplitz operator is given by 
\begin{equation}\label{e-gue210707ycd}
\begin{split}
&T_P:=S^{(q)}\circ P\circ S^{(q)}: L^2_{(0,q)}(X)\To{\rm Ker\,}\Box^q_b,\\
&T_{P,-}:=S_-\circ P\circ S_-: L^2_{(0,q)}(X)\To{\rm Ker\,}\Box^q_b,\\
&T_{P,+}:=S_+\circ P\circ S_+: L^2_{(0,q)}(X)\To{\rm Ker\,}\Box^q_b,
\end{split}
\end{equation}
where $S_-, S_+$ are as in \eqref{e-gue210707yyd}. 
If $q\neq n_+$, then $T_{P,+}\equiv0$ on $X$. 

Let $f\in\mathcal{C}^\infty(X)$ and let $M_f$ denote the operator given by the multiplication $f$. When $P=M_f$, we write $T_f:=T_P$. 
\end{defn}

The goal of this section is to study the algebra of $T_P$ . The following follows from the standard calculus of Fourier integral operator of complex type (see the calculation below) 

\begin{thm}\label{t-gue210701yyd}
Let $P\in L^{\ell}_{{\rm cl\,}}(X,T^{*0,q}X\boxtimes(T^{*0,q}X)^*)$ with scalar principal symbol, $\ell\leq0$, $\ell\in\mathbb Z$. We have $T_P, T_{P,\mp}\in\Psi_{n+\ell}(X)$ and 
\[\mbox{$\sigma^0_{T_{P,-}}(x,x)=\frac{1}{2}\pi^{-n-1}\abs{{\rm det\,}\mathcal L_x}\sigma^0_P(x,-\omega_0(x))\tau_{x,n_-}$, for every $x\in X$}.\]
If $q=n_-=n_+$, then
\[\mbox{$\sigma^0_{T_{P,+}}(x,x)=\frac{1}{2}\pi^{-n-1}\abs{{\rm det\,}\mathcal L_x}\sigma^0_P(x,\omega_0(x))\tau_{x,n_+}$, for every $x\in X$}.\]
\end{thm}

Let $x=(x_1,\ldots,x_{2n+1})$ be local coordinates of $X$ defined on an open set $D\subset X$. Fix $p\in D$. Take $x=(x_1,\ldots,x_{2n+1})$ so that $x(p)=0$ and \eqref{e-gue161219a} hold. We will use the same notations as in the proof of Lemma~\ref{l-gue210624yyd}. Let $P\in L^{\ell}_{{\rm cl\,}}(X,T^{*0,q}X\boxtimes(T^{*0,q}X)^*)$ with scalar principal symbol, $\ell\leq0$, $\ell\in\mathbb Z$. 
We write 
 \begin{equation}\label{e-gue210704yyd}
 \begin{split}
&\mbox{$T_{P,-}(x,y)\equiv\int^\infty_0e^{i\varphi_{-}(x, y)t}a_P(x, y, t)dt$ on $D$},\\
&a_P(x, y, t)\sim\sum^\infty_{j=0}a^j_P(x, y)t^{n-j+\ell}\text{ in }S^{n+\ell}_{1, 0}
(D\times D\times\mathbb{R}_+\,,T^{*0,q}X\boxtimes(T^{*0,q}X)^*),\\
&a^j_P(x, y)\in\mathcal{C}^\infty(D\times D,T^{*0,q}X\boxtimes(T^{*0,q}X)^*),\ j\in\mathbb N_0.
\end{split}\end{equation}
We assume that $\varphi_-$ is of the form \eqref{e-gue210701yyd} and 
$s^j_-(x,y)$, $a^j_P(x,y)$, $s_-(x,y,t)$, $a_P(x,y,t)$ are independent of $y_{2n+1}$, for every $j=0,1,\ldots$. Furthermore, we suppose that the projections 
\[\begin{split}
&(x,y)\in{\rm supp\,}a_P(x,y,t)\To x,\ \ (x,y)\in{\rm supp\,}a^j_P(x,y)\To x,\\
&(x,y)\in{\rm supp\,}s_-(x,y,t)\To x,\ \ (x,y)\in{\rm supp\,}s^j_-(x,y)\To x,
\end{split} \]
are proper, for every $j=0,1,\ldots$. Note that 
\begin{equation}\label{e-gue210719yyd}
(d_x\varphi_-)(x,x)=-\omega_0(x),\ \ \mbox{for every $x\in D$}. 
\end{equation}

 Let $\beta(\Td x,\Td y)=(\beta_1(\Td x,\Td y),\ldots,\beta_{2n+1}(\Td x,\Td y))\in\mathcal{C}^\infty(D^{\mathbb C}\times D^{\mathbb C},\mathbb C^{2n+1})$, 
$\gamma(\Td x,\Td y)\in\mathcal{C}^\infty(D^{\mathbb C}\times D^{\mathbb C},\mathbb C)$ be as in \eqref{e-gue210701yydII}. From \eqref{e-gue140205VII} and \eqref{e-gue210701yydII}, it is straightforward to check that 
\begin{equation}\label{e-gue210701ycd}
\begin{split}
&\frac{\pr\beta_{2j-1}}{\pr x_{2j}}(0,0)=\frac{\pr\beta_{2j}}{\pr y_{2j-1}}(0,0)=\frac{-\mu_j}{2i\abs{\mu_j}},\ \ j=1,\ldots,n,\\
&\frac{\pr\beta_{2j}}{\pr x_{2j-1}}(0,0)=\frac{\pr\beta_{2j-1}}{\pr y_{2j}}(0,0)=\frac{\mu_j}{2i\abs{\mu_j}},\ \ j=1,\ldots,n,\\
&\frac{\pr\beta_{2j-1}}{\pr x_{2j-1}}(0,0)=\frac{\pr\beta_{2j}}{\pr x_{2j}}(0,0)=\frac{\pr\beta_{2j-1}}{\pr y_{2j-1}}(0,0)=\frac{\pr\beta_{2j}}{\pr y_{2j}}(0,0)=\frac{1}{2},\ \ j=1,\ldots,n,\\
&\frac{\pr\beta_{2j}}{\pr x_{s}}(0,0)=\frac{\pr\beta_{2j}}{\pr y_{s}}(0,0)=0,\ \ s\notin\set{2j-1,2j},\ \ j=1,\ldots,n,\\
&\frac{\pr\beta_{2j-1}}{\pr x_{s}}(0,0)=\frac{\pr\beta_{2j-1}}{\pr y_{s}}(0,0)=0,\ \ s\notin\set{2j-1,2j},\ \ j=1,\ldots,n.
\end{split}
\end{equation}
Form complex stationary phase formula, we can check that 
\begin{equation}\label{e-gue140205VIIa}
a^0_P(x,y')=\Td{\sigma^0_P}(\beta(x,y'),\Td\varphi_{-,x}(\beta(x,y'),y))s^0_-(x,y')+O(\abs{x-y}^N), 
\end{equation}
for every $N\in\mathbb N$.
Let $F(\Td x,\Td y,\Td\sigma,\Td u):=\Td\varphi_-(\Td x,\Td u)\Td\sigma+\Td\varphi_-(\Td u,\Td y)$. 
Fix $D_0\Subset D$, $D_0$ is an open set of 
$0$. Let $\chi\in\mathcal{C}^{\infty}_c(\mathbb R,[0,1])$, $\chi\equiv1$ on $[-\frac{1}{2},\frac{1}{2}]$. Let $\varepsilon>0$ be a small constant so that 
$\frac{x_{2n+1}-y_{2n+1}}{\varepsilon}\notin{\rm supp\,}\chi$, for every $(x',x_{2n+1})\in D_0$ and every $(y',y_{2n+1})\notin D$. As in the proof of Lemma~\ref{l-gue210624yyd}, we have 
\begin{equation}\label{e-gue210701ycdI}
\begin{split}
&\int^\infty_0e^{i\varphi_{-}(x, y)t}a_P(x, y, t)dt\\
&\equiv\int^\infty_0\int_X \int^\infty_0 e^{it(\varphi_{-}(x, u)\sigma+i\varphi_-(u,y))}\chi(\frac{1-\sigma}{\varepsilon})\chi(\frac{x_{2n+1}-u_{2n+1}}{\varepsilon})ta_P(x,u',t\sigma)s_-(u,y',t)d\sigma dv(u)dt\\
&\equiv\int^\infty_0e^{i\varphi_{-}(x, y)t}g(x, y, t)dt,
\end{split}
\end{equation}
where 
$g(x, y, t)\sim\sum^\infty_{j=0}g_j(x, y)t^{n-j+\ell}\text{ in }S^{n+\ell}_{1, 0}
(D_0\times D_0\times\mathbb{R}_+\,,T^{*0,q}X\boxtimes(T^{*0,q}X)^*)$, 
$g_j(x, y)\in\mathcal{C}^\infty(D_0\times D_0,T^{*0,q}X\boxtimes(T^{*0,q}X)^*)$, $j\in\mathbb N_0$, 
$g_j(x, y)$ is independent of $y_{2n+1}$, for every $j\in\mathbb N_0$, 
\begin{equation}\label{e-gue210701ycdII}
g_0(x,y)=\Bigr({\rm det\,}(\frac{F''_{\Td\sigma,\Td u}}{2\pi\imath})\Bigr)^{-\frac{1}{2}}|_{\Td u=\beta(x,y),\Td\sigma=\gamma(x,y)}\Td a^0_P(x,\beta(x,y'))\Td s^-_0(\beta(x,y'),y')+O(\abs{x-y}^N),
\end{equation}
\begin{equation}\label{e-gue210702yyd}
\begin{split}
&\sum^N_{j=0}t^{n-j+\ell}g_j(x,y)\\
&-\Bigr({\rm det\,}(\frac{F''_{\Td\sigma,\Td u}}{2\pi i})\Bigr)^{-\frac{1}{2}}|_{\Td u=\beta(x,y),\Td\sigma=\gamma(x,y)}
\sum^N_{j=0}L_{j,x,y}\Bigr(\Td a_P(x, \Td u',\Td\sigma t)\Td s_-(\Td u,y',t)\Td v_1\Bigr)|_{\Td u=\beta(x,y),\Td\sigma=\gamma(x,y)}t^{-n-j+\ell}\\
&\in S^{n-N-1+\ell}_{1, 0}
(D_0\times D_0\times\mathbb{R}_+\,,T^{*0,q}X\boxtimes(T^{*0,q}X)^*),
\end{split}
\end{equation}
for every $N\in\mathbb N$, where $v_1=v_1(x,u,\sigma)=\chi(\frac{1-\sigma}{\varepsilon})\chi(\frac{x_{2n+1}-u_{2n+1}}{\varepsilon})v(u)$, $v(u)du=dv(u)$, $L_{j,x,y}$ is a differential operator in $u$ of order $\leq 2j$ which is a $\mathcal{C}^\infty$ function of $(x,y')$, for 
every $j=0,1,2,\ldots$, $F''_{\Td\sigma,\Td u}=\left(\frac{\pr^2F}{\pr\Td u_j\pr\Td u_k}\right)^{2n+1}_{j,k=0}$, $\Td u_0=\Td\sigma$. From \eqref{e-gue210701ycdI} and notice that $a^j_P(x,y)$, $g_j(x,y)$ are independent of $y_{2n+1}$, for every $j=0,1,\ldots$, we conclude that 
\begin{equation}\label{e-gue210702yydI}
\mbox{$a^j_P(x,y)=g_j(x,y)+O(\abs{x-y}^N)$, for every $N\in\mathbb N$, $j=0,1,\ldots$}. 
\end{equation}
Now, take $x=y=0$. It is straightforward to check that 
\begin{equation}\label{e-gue210702ycd}
\begin{split}
&\beta(0,0)=0,\ \ \gamma(0,0)=1,\\
&\Bigr({\rm det\,}(\frac{F''_{\Td\sigma,\Td u}(0,0,1,0)}{2\pi\imath})\Bigr)^{-\frac{1}{2}}
=2^{-n+1}\pi^{n+1}\abs{\mu_1}^{-1}\cdots\abs{\mu_n}^{-1},
\end{split}
\end{equation}
and 
\begin{equation}\label{e-gue210703ycd}
a^0_P(0,0)=\frac{1}{2}\pi^{-n-1}\abs{{\rm det\,}\mathcal{L}_p}\sigma^0_P(p,-\omega_0(p))\tau_{p,n_-}.
\end{equation}
From \eqref{e-gue140205VII}, \eqref{e-gue210702yyd}, \eqref{e-gue210702yydI}, H\"ormander's formula~\cite[Theorem 7.7.5]{hor} and notice 
that $v(0)=2^n$, we conclude that 
\begin{equation}\label{e-gue210702ycdI}
\begin{split}
&a^1_P(0,0)=2^{-n+1}\pi^{n+1}\abs{\mu_1}^{-1}\cdots\abs{\mu_n}^{-1}\Bigr(a^1_P(0,0)s^0_-(0,0)v(0)+s^1_-(0,0)a^0_P(0,0)v(0)\\
&+\sum_{\nu-\mu=1}\sum_{2\nu \geq 3\mu}\frac{\imath^{-1}2^{-\nu}}{\nu!\,\mu!}L^{\nu}(h(\sigma,u)^{\mu}\Td a^0_P(0,u)v_1(0,u,\sigma)\sigma^ns^0_-(u,0))|_{u=0,\sigma=1}\Bigr),\end{split}\end{equation}
where 
\begin{equation}\label{e-gue210703yyd}
h(\sigma,u):=F(0,0,\sigma,u)-F(0,0,1,0)-\frac{1}{2}\langle\,F''_{\Td\sigma,\Td u}(0,0,1,0)(\sigma,u)^t,(\sigma,u)^t\,\rangle,
\end{equation}
\begin{equation}\label{e-gue210703yydI}
	L= -2\,\partial_{u_{2n+1}}\partial_{\sigma}+\imath\, \sum_{j=1}^n \frac{1}{\abs{\mu_j}}\pr_{\ol z_j}\pr_{z_j}.
	\end{equation}
	Let 
	\begin{equation}\label{e-gue210721yyd}
	f(u):=\Td{\sigma^0_P}(\beta(0,u'),\Td\varphi_{-,x}(\beta(0,u'),y)).
	\end{equation}
Suppose that $\mu_j<0$, $j=1,\ldots,n_-$, $\mu_j>0$, $j=n_-+1,\ldots,n$. We will use the same notations as in the discussion after Theorem~\ref{t-gue161110g}. Compare the two side of \eqref{e-gue210702ycdI} in the direction of $e_{I_0}\otimes(e_{I_0})^*$ and using Theorem~\ref{t-gue140205III} and 
\eqref{e-gue210703ycd}, we get 
\begin{equation}\label{e-gue210703yyda}
\begin{split}
&a^1_{P,I_0,I_0}(0,0)\\
&=a^1_{P,I_0,I_0}(0,0)+s^1_{-,I_0,I_0}(0,0)\sigma^0_P(0,-\omega_0(0))\\
&\quad+C_0\sum_{\nu-\mu=1}\sum_{2\nu \geq 3\mu}\frac{\imath^{-1}2^{-\nu}}{\nu!\,\mu!}L^{\nu}(h(\sigma,u)^{\mu} s^0_{-,I_0,I_0}(0,u)f(u)v_1(0,u,\sigma)\sigma^ns^0_{-,I_0,I_0}(u,0))|_{u=0,\sigma=1}\\
&=a^1_{P,I_0,I_0}(0,0)+s^1_{-,I_0,I_0}(0,0)\sigma^0_P(0,-\omega_0(0))\\
&\quad+\sigma^0_P(0,-\omega_0(0))C_0\sum_{\nu-\mu=1}\sum_{2\nu \geq 3\mu}\frac{\imath^{-1}2^{-\nu}}{\nu!\,\mu!}L^{\nu}(h(\sigma,u)^{\mu} s^0_{-,I_0,I_0}(0,u)v_1(0,u,\sigma)\sigma^ns^0_{-,I_0,I_0}(u,0))|_{u=0,\sigma=1}\\
&\quad+\varepsilon_P,
\end{split}
\end{equation}
where $C_0=2^{-n+1}\pi^{n+1}\abs{\mu_1}^{-1}\cdots\abs{\mu_n}^{-1}$ and 
\begin{equation}\label{e-gue210703yydb}
\begin{split}
\varepsilon_P=&\sum^{N_0}_{j=1}\imath^{-1}2^{-1}Z_j(f(u))|_{u=0}\ell_j\Bigr( s^0_{-,I_0,I_0}(0,u)\sigma^nv_1(0,u,\sigma)s^0_{-,I_0,I_0}(u,0)\Bigr)|_{u=0,\sigma=1}\\
&\quad+\sum^{N_1}_{j=1}\imath^{-1}\frac{1}{8}\hat Z_j(f(u))|_{u=0}P_j
\Bigr(h(\sigma,u)s^0_{-,I_0,I_0}(0,u)\sigma^nv_1(0,u,\sigma)s^0_{-,I_0,I_0}(u,0)\Bigr)|_{u=0,\sigma=1}\\
&\quad+\imath^{-1}2^{-1}L(f(u))|_{u=0}s^2_{-,I_0,I_0}(0,0)2^n,
\end{split}
\end{equation}
$Z_j$, $\ell_j$ are differential operators of constant coefficients of order $1$, $\hat Z_k$, $P_k$ are differential operators of constant coefficients of order $1$ and order $3$ respectively, $j=1,\ldots,N_0$, $k=1,\ldots,N_1$, $N_0, N_1\in\mathbb N$. 

Take $P=I$ in \eqref{e-gue210703yyda}, we get 
\begin{equation}\label{e-gue210705yyd}
\begin{split}
&s^1_{-,I_0,I_0}(0,0)\\
&=-C_0\sum_{\nu-\mu=1}\sum_{2\nu \geq 3\mu}\frac{\imath^{-1}2^{-\nu}}{\nu!\,\mu!}L^{\nu}(h(\sigma,u)^{\mu} s^0_{-,I_0,I_0}(0,u)v_1(0,u,\sigma)\sigma^ns^0_{-,I_0,I_0}(u,0))|_{u=0,\sigma=1}.
\end{split}
\end{equation}
From \eqref{e-gue210703yyda} and \eqref{e-gue210705yyd}, we get 
\begin{equation}\label{e-gue210705yydI}
\varepsilon_P=0. 
\end{equation}

We introduce some notations. 
For $u\in\mathcal{C}^\infty(X)$, there is a unique vector field $X_u$ such that 
\begin{equation}\label{e-gue210706yyd}
\iota(X_u)\omega_0=u,\ \ \iota(X_u)d\omega_0=(\iota(-R)du)\omega_0-du.  
\end{equation}
For $u, v\in\mathcal{C}^\infty(X)$, the transversal Poisson bracket is defined by  
\begin{equation}\label{e-gue210706yydI}
\{u, v\}:=d\omega_0(X_{u},X_{v}). 
\end{equation}
In the local coordinates $x=(x_1,\ldots,x_{2n+1})$ satisfying $x(p)=0$ and \eqref{e-gue161219a}, it is straightforward to check that 
\begin{equation}\label{e-gue210706yydII}
X_{u}(p)=\sum^n_{j=1}\Bigr(\frac{1}{2\imath\mu_j}\frac{\pr u}{\pr\ol z_j}(p)\frac{\pr}{\pr z_j}+\frac{1}{-2\imath\mu_j}\frac{\pr u}{\pr z_j}(p)\frac{\pr}{\pr\ol z_j}\Bigr)+u\frac{\pr}{\pr x_{2n+1}}
\end{equation}
and 
\begin{equation}\label{e-gue210706yydIII}
\{u, v\}(p)=(-\imath)\sum^n_{j=1}\Bigr(\frac{1}{2\mu_j}\frac{\pr u}{\pr\ol z_j}(p)\frac{\pr v}{\pr z_j}(p)-\frac{1}{2\mu_j}\frac{\pr u}{\pr z_j}(p)\frac{\pr v}{\pr\ol z_j}(p)\Bigr).
\end{equation} 

Let 
\begin{equation}\label{e-gue210722yyd}
\triangle_b:=\Box^q_b+R^*R: \Omega^{0,q}(X)\To\Omega^{0,q}(X),
\end{equation}
where $R^*$ is the formal adjoint of $R$ with respect to $(\,\cdot\,|\,\cdot\,)$. For every $j\in\mathbb Z$, $j\leq0$, put 
\begin{equation}\label{e-gue210722yydI}
\hat S^j\:=\set{f(x)(\sigma^0_{\triangle_b})^{\frac{j}{2}}\in\mathcal{C}^\infty(TX);\, f(x)\in\mathcal{C}^\infty(X)}
\end{equation}
and let 
\begin{equation}\label{e-gue210722yydII}
\hat S:=\bigcup_{j\in\mathbb Z, j\leq0}\hat S^j.
\end{equation}

For $a\in S^{k}_{1,0}(X)$, $b\in S^{\ell}_{1,0}(X)$, $k, \ell\in\mathbb Z$, $k, \ell\leq0$, the transversal Poisson bracket  of $a$ and $b$ with respect to $\Sigma^-$
 is defined to be the element $\{a,b\}_-\in\hat S^{k+\ell-1}$ such that 
\begin{equation}\label{e-gue210721yyda}
\{a,b\}_-(x,-\omega_0(x))=\{\hat a_-, \hat b_-\}(x),\ \ \mbox{for every $x\in X$},
\end{equation}
where $\hat a_-(x):=a(x,-\omega_0(x))\in\mathcal{C}^\infty(X)$, $\hat b_-(x):=b(x,-\omega_0(x))\in\mathcal{C}^\infty(X)$. 
Note that $\{a,b\}_-$ is uniquely determined by \eqref{e-gue210721yyda}. 
Similarly, the transversal Poisson bracket  of $a$ and $b$ with respect to $\Sigma^+$
 is defined to be the element $\{a,b\}_+\in\hat S^{k+\ell-1}$ such that 
 \begin{equation}\label{e-gue210722yydIII}
\{a,b\}_+(x,\omega_0(x))=\{\hat a_+, \hat b_+\}(x),\ \ \mbox{for every $x\in X$},
\end{equation}
 where $\hat a_+(x):=a(x,\omega_0(x))\in\mathcal{C}^\infty(X)$, $\hat b_+(x):=b(x,\omega_0(x))\in\mathcal{C}^\infty(X)$. 
 
 In the local coordinates $x=(x_1,\ldots,x_{2n+1})$ satisfying $x(p)=0$ and \eqref{e-gue161219a}, it is straightforward to check that 
\begin{equation}\label{e-gue210721yydb}
\begin{split}
&\{a,b\}_-(p,-\omega_0(p))\\
&=(-\imath)\sum^n_{j=1}\Bigr(\frac{1}{2\mu_j}\bigr(\frac{\pr a}{\pr\ol z_j}(p,-\omega_0(p))-2i\mu_j\frac{\pr a}{\pr\ol\theta_j}(p,-\omega_0(p))\bigr)
\bigr(\frac{\pr b}{\pr z_j}(p,-\omega_0(p))+2i\mu_j\frac{\pr b}{\pr\theta_j}(p,-\omega_0(p))\bigr)\\
&\quad-\frac{1}{2\mu_j}\bigr(\frac{\pr a}{\pr z_j}(p,-\omega_0(p))+2i\mu_j\frac{\pr a}{\pr\theta_j}(p,-\omega_0(p))\bigr)
\bigr(\frac{\pr b}{\pr\ol z_j}(p,-\omega_0(p))-2i\mu_j\frac{\pr b}{\pr\ol\theta_j}(p,-\omega_0(p))\bigr)\Bigr),
\end{split}
\end{equation} 
where $\frac{\pr}{\pr\theta_j}=\frac{1}{2}(\frac{\pr}{\pr\eta_{2j-1}}-i\frac{\pr}{\pr\eta_{2j}})$, $\frac{\pr}{\pr\ol\theta_j}=\frac{1}{2}(\frac{\pr}{\pr\eta_{2j-1}}+i\frac{\pr}{\pr\eta_{2j}})$, 
$j=1,\ldots,n$.

We can now prove 

\begin{thm}\label{t-gue210706yyd}
Let $P\in L^{\ell}_{{\rm cl\,}}(X,T^{*0,q}X\boxtimes(T^{*0,q}X)^*)$ with scalar principal symbol, and let $Q\in L^{k}_{{\rm cl\,}}(X,T^{*0,q}X\boxtimes(T^{*0,q}X)^*)$ with scalar principal symbol, $\ell, k\leq0$, $\ell, k\in\mathbb Z$. 
We have $[T_P, T_Q]\in\Psi_{n-1+\ell+k}(X)$ and 
\begin{equation}\label{e-gue210706ycd}
\mbox{$\tau_{x,n_-}\sigma^0_{[T_P, T_Q],-}(x,x)\tau_{x,n_-}=\frac{1}{2}\pi^{-n-1}\abs{\det\mathcal{L}_x}\imath\{\sigma^0_P,\sigma^0_Q\}_-(x,-\omega_0(x))\tau_{x,n_-}$, for every $x\in X$}, 
\end{equation}
and if $q=n_-=n_+$, 
\begin{equation}\label{e-gue210706ycdI}
\mbox{$\tau_{x,n_+}\sigma^0_{[T_P, T_Q],+}(x,x)\tau_{x,n_+}=\frac{1}{2}\pi^{-n-1}\abs{\det\mathcal{L}_x}\imath\{\sigma^0_P,\sigma^0_Q\}_+(x,\omega_0(x))\tau_{x,n_+}$, for every $x\in X$}. 
\end{equation}
\end{thm}

\begin{proof}
We first assume that $q\neq n_+$. Let $x=(x_1,\ldots,x_{2n+1})$ be local coordinates of $X$ defined on an open set $D\subset X$. Fix $p\in D$. Take $x=(x_1,\ldots,x_{2n+1})$ so that $x(p)=0$ and \eqref{e-gue161219a} hold. We will use the same notations and assumptions as in the calculation before. From the calculation above, we have
 \begin{equation}\label{e-gue210706ycdII}
 \begin{split}
&\mbox{$(T_P\circ T_Q)(x,y)\equiv\int^\infty_0e^{i\varphi_{-}(x, y)t}a_{P,Q}(x, y, t)dt$ on $D$},\\
&a_{P,Q}(x, y, t)\sim\sum^\infty_{j=0}a^j_{P,Q}(x, y)t^{n-j+\ell+k}\text{ in }S^{n+\ell+k}_{1, 0}
(D\times D\times\mathbb{R}_+\,,T^{*0,q}X\boxtimes(T^{*0,q}X)^*),\\
&a^j_{P,Q}(x, y)\in\mathcal{C}^\infty(D\times D,T^{*0,q}X\boxtimes(T^{*0,q}X)^*),\ j\in\mathbb N_0, 
\end{split}\end{equation}
and 
\begin{equation}\label{e-gue210706ycdIII}
a^0_{P,Q}(x,y)=\Td\sigma^0_{T_P,-}(x,\beta(x,y))\Td\sigma^0_{T_Q,-}(\beta(x,y),y)+O(\abs{x-y}^N),
\end{equation}
for every $N\in\mathbb N$.
From \eqref{e-gue210706ycdIII} and notice that $\beta(x,x)=x$, we get 
\begin{equation}\label{e-gue210706yyda}
\mbox{$a^0_{P,Q}(x,x)-a^0_{Q,P}(x,x)=0$, for every $x\in D$}. 
\end{equation}

From Lemma~\ref{l-gue210624yyd} and \eqref{e-gue210706yyda}, we conclude that $[T_P,T_Q]\in\Psi_{n-1+\ell+k}(X)$. Put 
\begin{equation}\label{e-gue210723yyd}
\begin{split}
&f(u):=\Td{\sigma^0_P}(\beta(0,u'),\Td\varphi_{-,x}(\beta(0,u'),0)),\\
&g(u):=\Td{\sigma^0_Q}(\beta(u,0),\Td\varphi_{-,x}(\beta(u,0),0)).
\end{split}
\end{equation}
We can repeat the calculation before with minor change and conclude that 
\begin{equation}\label{e-gue210706yydb}
\begin{split}
&a^1_{P,Q,I_0,I_0}(p,p)\\
&=a^1_{P,I_0,I_0}(p,p)a^0_{Q,I_0,I_0}(p,p)+a^0_{P,I_0,I_0}(p,p)a^1_{Q,I_0,I_0}(p,p)\\
&\quad +f(p)g(p)C_0\sum_{\nu-\mu=1}\sum_{2\nu \geq 3\mu}\frac{\imath^{-1}2^{-\nu}}{\nu!\,\mu!}L^{\nu}(h(\sigma,u)^{\mu} s^0_{-,I_0,I_0}(0,u)v_1(0,u,\sigma)\sigma^ns^0_{-,I_0,I_0}(u,0))|_{u=0,\sigma=1}\\
&\quad+g(p)\varepsilon_P+f(p)\hat \varepsilon_Q\\
&\quad+\frac{1}{2}
\pi^{-n-1}
\abs{{\rm det\,}\mathcal L_p}\frac{1}{2}\sum^n_{j=1}\frac{1}{\abs{\mu_j}}\Bigr(\frac{\pr}{\pr\ol z_j}(f(u))|_{u=0}\frac{\pr}{\pr z_j}(g(u))|_{u=0}+\frac{\pr}{\pr z_j}(f(u))|_{u=0}\frac{\pr}{\pr\ol z_j}(g(u))|_{u=0}\Bigr),
\end{split}
\end{equation}
where $C_0=2^{-n+1}\pi^{n+1}\abs{\mu_1}^{-1}\cdots\abs{\mu_n}^{-1}$, $\varepsilon_P$ is given by \eqref{e-gue210703yydb} and 
\begin{equation}\label{e-gue210706yydc}
\begin{split}
\hat \varepsilon_Q=&\sum^{N_0}_{j=1}\imath^{-1}2^{-1}Z_j(g(u))|_{u=0}\ell_j\Bigr( s^0_{-,I_0,I_0}(0,u)\sigma^nv_1(0,u,\sigma)s^0_{-,I_0,I_0}(u,0)\Bigr)|_{u=0,\sigma=1}\\
&\quad+\sum^{N_1}_{j=1}\imath^{-1}\frac{1}{8}\hat Z_j(g(u))|_{u=0}P_j
\Bigr(h(\sigma,u)s^0_{-,I_0,I_0}(0,u)\sigma^nv_1(0,u,\sigma) s^0_{-,I_0,I_0}(u,0)\Bigr)|_{u=0,\sigma=1}\\
&\quad+\imath^{-1}2^{-1}L(g(u))|_{u=0}s^2_{-,I_0,I_0}(0,0)2^n,
\end{split}
\end{equation}
$Z_j$, $\ell_j$, $\hat Z_k$, $P_k$ are as in \eqref{e-gue210703yydb}, $j=1,\ldots,N_0$, $k=1,\ldots,N_1$, $N_0, N_1\in\mathbb N$.
From \eqref{e-gue210705yydI}, we know that $\varepsilon_P=0$. Similar, we can repeat the proof of \eqref{e-gue210705yydI} and conclude hat 
\begin{equation}\label{e-gue210706yydd}
\hat\varepsilon_Q=0.
\end{equation}
From \eqref{e-gue210701ycd} and \eqref{e-gue210723yyd}, we can check that 
\begin{equation}\label{e-gue210706yyde}
\begin{split}
&\frac{\pr}{\pr z_j}(g(u))|_{u=0}=\frac{\pr}{\pr\ol z_j}(f(u))|_{z=0}=0,\ \ j\in\set{1,\ldots,q},\\
&\frac{\pr}{\pr\ol z_j}(f(u))|_{u=0}=\Bigr(\frac{\pr\sigma^0_P}{\pr\ol z_j}-2i\abs{\mu_j}\frac{\pr\sigma^0_P}{\pr\ol\theta_j}\Bigr)(p,-\omega_0(p)),\ \ j\in\set{q+1,\ldots,n},\\
&\frac{\pr}{\pr z_j}(g(u))|_{u=0}=\Bigr(\frac{\pr\sigma^0_Q}{\pr z_j}+2i\abs{\mu_j}\frac{\pr\sigma^0_Q}{\pr\theta_j}\Bigr)(p,-\omega_0(p)),\ \ j\in\set{q+1,\ldots,n},\\
&\frac{\pr}{\pr\ol z_j}(g(u))|_{u=0}=\frac{\pr}{\pr z_j}(f(u))|_{u=0}=0,\ \ j\in\set{q+1,\ldots,n},\\
&\frac{\pr}{\pr z_j}(f(u))|_{u=0}=\Bigr(\frac{\pr\sigma^0_P}{\pr z_j}-2i\abs{\mu_j}\frac{\pr\sigma^0_P}{\pr\theta_j}\Bigr)(p,-\omega_0(p)),\ \ j\in\set{1,\ldots,q},\\
&\frac{\pr}{\pr\ol z_j}(g(u))|_{u=0}=\Bigr(\frac{\pr\sigma^0_Q}{\pr\ol z_j}+2i\abs{\mu_j}\frac{\pr\sigma^0_Q}{\pr\ol\theta_j}\Bigr)(p,-\omega_0(p)),\ \ j\in\set{1,\ldots,q},
\end{split}
\end{equation}
where $\frac{\pr}{\pr\theta_j}=\frac{1}{2}(\frac{\pr}{\pr\eta_{2j-1}}-i\frac{\pr}{\pr\eta_{2j-1}})$, $\frac{\pr}{\pr\ol\theta_j}=\frac{1}{2}(\frac{\pr}{\pr\eta_{2j-1}}+i\frac{\pr}{\pr\eta_{2j-1}})$, $j=1,\ldots,n$. 
From \eqref{e-gue210705yydI}, \eqref{e-gue210706yydb}, \eqref{e-gue210706yydd} and \eqref{e-gue210706yyde}, we deduce that 
\begin{equation}\label{e-gue210706yydp}
\begin{split}
&a^1_{P,Q,I_0,I_0}(p,p)\\
&=a^1_{P,I_0,I_0}(p,p)a^0_{Q,I_0,I_0}(p,p)+a^0_{P,I_0,I_0}(p,p)a^1_{Q,I_0,I_0}(p,p)\\
&\quad +f(p)g(p)C_0\sum_{\nu-\mu=1}\sum_{2\nu \geq 3\mu}\frac{\imath^{-1}2^{-\nu}}{\nu!\,\mu!}L^{\nu}(h(\sigma,u)^{\mu}\Td s^0_{-,I_0,I_0}(0,u)v(u)\sigma^n\Td s^0_-(u,0))|_{u=0,\sigma=1}\\
&\quad+\frac{1}{2}\pi^{-n-1}\abs{{\rm det\,}\mathcal{L}_p}\frac{1}{2}\sum^n_{j=q+1}\frac{1}{\abs{\mu_j}}\Bigr(\frac{\pr f}{\pr\ol z_j}(0)\frac{\pr g}{\pr z_j}(0)\Bigr)+\sum^q_{j=1}\frac{1}{\abs{\mu_j}}\Bigr(\frac{\pr f}{\pr z_j}(0)\frac{\pr g}{\pr\ol z_j}(0)\Bigr).
\end{split}
\end{equation}
From \eqref{e-gue210721yydb}, \eqref{e-gue210706yyde} and \eqref{e-gue210706yydp}, it is straightforward to check that 
\begin{equation}\label{e-gue210706yydz}
\begin{split}
&a^1_{P,Q,I_0,I_0}(p,p)-a^1_{Q,P,I_0,I_0}(p,p)\\
&=\frac{1}{2}\pi^{-n-1}\abs{{\rm det\,}\mathcal{L}_p}\imath\{\sigma^0_P,\sigma^0_Q\}_-(x,-\omega_0(x)).
\end{split}
\end{equation}
We get the theorem if $q\neq n_+$. 

Assume that $q=n_-=n_+$. From~\cite[Lemma 5.6]{HM14}, we see that $T_{P,-}\circ T_{Q,+}\equiv0$,  $T_{P,+}\circ T_{Q,-}\equiv0$. From this observation, we can repeat the procedure above and get the theorem. 
\end{proof} 

From Lemma~\ref{l-gue210624yyd} and Theorem~\ref{t-gue210706yyd}, we get 

\begin{thm}\label{t-gue210724yyd}
Let $P\in L^{\ell}_{{\rm cl\,}}(X,T^{*0,q}X\boxtimes(T^{*0,q}X)^*)$ with scalar principal symbol, and let $Q\in L^{k}_{{\rm cl\,}}(X,T^{*0,q}X\boxtimes(T^{*0,q}X)^*)$ with scalar principal symbol, $\ell, k\leq0$, $\ell, k\in\mathbb Z$. 
We have 
\begin{equation}\label{e-gue210724yyda}\begin{split}
&[T_{P,-}, T_{Q,-}]-T_{{\rm Op\,}(\imath\{\sigma^0_P,\sigma^0_Q\}_-),-} \in\Psi_{n-2+\ell+k}(X),\\
&[T_{P,+}, T_{Q,+}]-T_{{\rm Op\,}(\imath\{\sigma^0_P,\sigma^0_Q\}_+),+} \in\Psi_{n-2+\ell+k}(X).
\end{split}\end{equation}
\end{thm}

Consider $\hat S$ (recall that $\hat S$ is given by \eqref{e-gue210722yydII}). For $a, b\in\hat S$, define 
\begin{equation}\label{e-gue210724yyd}
a\hat{+}b:=\mbox{leading term of $a+b$}. 
\end{equation}
For $a, b\in\hat S$, then $a\hat{+}b\in\hat S^j$, for some $j\in\mathbb Z$, $j\leq0$. We can identify $\hat S^j$ with all homogeneous functions on $\Sigma$ of degree $j$. 
Then, $\set{\hat S,\hat{+}}$ is a vector space and also $\hat S$ has natural algebraic structure, that is, if $a, b\in\hat S$, then $a.b\in\hat S$. 

A star product for the algebra of $\hat S$ with respect to $\Sigma^{\mp}$ is given by the power series \[a*b= \sum_{j=0}^{+\infty} C_{j,\mp}(a,b) \nu^{-j}  \]
such that $*$ is an associative $\mathbb C[[\nu]]$-linear product, that is, $(a*b)*c=a*(b*c)$, for all $a, b, c\in\hat S$ and $C_{0,\mp}(a,b)=a\cdot b$, $C_{1,\mp}(a,b)-C_{1,\mp}(b,a)=\imath\{a, b\}_{\mp}$, for all $a, b\in\hat S$. To simplify the notations, for $a\in\hat S$, we denote 
\begin{equation}\label{e-gue210724ycd}
T_a:=T_{{\rm Op\,}(a)},\ \ T_{a,-}:=T_{{\rm Op\,}(a),-},\ \ T_{a,+}:=T_{{\rm Op\,}(a),+}.
\end{equation}

We can now establish star product for $\hat S^j$.

\begin{thm}\label{t-gue210724yyda}
Recall that we work with the assumption that $q=n_-$. 
Let $a\in\hat S^\ell$, $b\in\hat S^k$, $\ell, k\in\mathbb Z$, $\ell, k\leq0$. We have
\begin{equation}\label{e-gue210724ycdI}
T_{a,-}\circ T_{b,-}-\sum^N_{j=0}T_{C_{j,-}(a,b),-}\in\Psi_{n-N-1+\ell+k}(X),
\end{equation}
for every $N\in\mathbb N_0$, where $C_{j,-}(a,b)\in\hat S^{\ell+k-j}$, $C_{j,-}$ is a universal bidifferential operator of order $\leq 2j$, $j=0,1,\ldots$, and 
\begin{equation}\label{e-gue210724ycdII}
\begin{split}
C_{0,-}(a,b)=ab,\\
C_{1,-}(a,b)-C_{1,-}(b,a)=\imath\{a,b\}_-.
\end{split}
\end{equation}

If $q=n_-=n_+$, we have
\begin{equation}\label{e-gue210724ycdIz}
T_{a,+}\circ T_{b,+}-\sum^N_{j=0}T_{C_{j,+}(a,b),\mp}\in\Psi_{n-N-1+\ell+k}(X),
\end{equation}
for every $N\in\mathbb N_0$, where $C_{j,+}(a,b)\in\hat S^{\ell+k-j}$, $C_{j,\mp}$ is a universal bidifferential operator of order $\leq 2j$, $j=0,1,\ldots$, and 
\begin{equation}\label{e-gue210724ycdIIz}
\begin{split}
C_{0,+}(a,b)=ab,\\
C_{1,+}(a,b)-C_{1,+}(b,a)=\imath\{a,b\}_{+}.
\end{split}
\end{equation}
\end{thm}

\begin{proof}
Let $C_{0,-}(a,b)=ab$. Let $B_0=T_{a,-}\circ T_{b,-}-T_{C_{0,-}(a,b),-}$. From \eqref{e-gue210706ycdIII}, we see that 
\begin{equation}\label{e-gue210707yydsz}
\sigma^0_{T_{a,-}\circ T_{b,-}}(x,x)=\sigma^0_{T_{C_{0,-}(a,b)},-}(x,x)\ \ \mbox{for every $x\in X$}.
\end{equation}
From Lemma~\ref{l-gue210624yyd}, we see that $B_0\in\Psi_{n-1+\ell+k}(X)$. Let $C_{1,-}(a,b)\in\hat S^{-1+\ell+k}$ so that 
\begin{equation}\label{e-gue210707yydtz}
\tau_{x,n_-}\sigma^0_{B_0,-}(x,x)\tau_{x,n_-}=\tau_{x,n_-}C_{1,-}(a,b)(x,-\omega_0(x))s^0_-(x,x)\tau_{x,n_-}\ \ \mbox{for every $x\in X$}.
\end{equation}
From complex stationary phase formula, we see that $C_1$ is a bidifferential operator of order $\leq 2$. From Lemma~\ref{l-gue210624yyd}, we see that 
\[T_{a,-}\circ T_{b,-}-\sum^1_{j=0}T_{C_{j,-}(a,b),-}\in\Psi_{n-2+\ell+k}(X).\]
Continuing in this way we get $C_{j,-}(a,b)\in\hat S^{-j+\ell+k}$, $C_j$ is a universal bidifferential operator of order $\leq 2j$, $j=0,1,\ldots$, such that \eqref{e-gue210724ycdI} holds. 

Now, from \eqref{e-gue210724ycdI}, we have 
\begin{equation}\label{e-gue210707yydwz}
[T_{a,-}, T_{b,-}]-T_{C_{1,-}(a,b)-C_{1,-}(b,a),-}\in\Psi_{n-2+\ell+k}(X).
\end{equation}
From \eqref{e-gue210724yyda} and \eqref{e-gue210707yydwz}, we get $C_{1,-}(a,b)-C_{1,-}(b,a)=\imath\{a,b\}_-$. The theorem follows. 
\end{proof}

\begin{thm}\label{t-gue210724yydc}
The star product
 \[a*b= \sum_{j=0}^{+\infty} C_{j,\mp}(a,b) \nu^{-j} ,\]
 $a, b\in\hat S$, is associative.
\end{thm}

\begin{proof}
For simplicity, we assume that $q\neq n_+$. The proof for $q=n_-=n_+$ is similar. Let $a, b, c\in\hat S$. For simplicity, assume that 
$a, b, c\in\hat S^0$.
Notice the associativity is equivalent to the equality
	\begin{equation}\label{e-gue210724yyds}
	\sum_{\ell=0}^k C_{\ell,-}(a,C_{k-\ell,-}(b,c))= \sum_{l=0}^k C_{\ell,-}(C_{k-\ell,-}(a,b),c) , 
	\end{equation}
	for every $k\in\mathbb N_0$. It is clear that \eqref{e-gue210724yyds} holds for $k=0$. Assume that 
	\eqref{e-gue210724yyds} holds for $k\leq k_0$, for some $k_0\in\mathbb N_0$. From 
	\[(T_{a,-}\circ T_{b,-})\circ T_{c,-}=T_{a,-}\circ (T_{b,-}\circ T_{c,-})\]
	and induction assumption, it is straightforward to check that 
	\begin{equation}\label{e-gue210724yydt}
	T_{\sum_{\ell=0}^{k_0+1} C_{\ell,-}(a,C_{k_0+1-\ell,-}(b,c)),-}-T_{\sum_{\ell=0}^{k_0+1} C_{\ell,-}(C_{k_0+1-\ell,-}(a,b),c),-}\in\Psi_{n-k_0-2}(X).
	\end{equation}
	From \eqref{e-gue210724yydt}, we get \eqref{e-gue210724yyds} for $k=k_0+1$. The theorem follows.  
\end{proof}

Let 
\begin{equation}\label{e-gue210706ycdg}
\hat R:=\frac{1}{2}S^{(q)}(-iR+(-iR)^*)S^{(q)}: \Omega^{0,q}(X)\To\Omega^{0,q}(X),
\end{equation}
where $(-iR)^*$ is the adjoint of $-iR$ with respect to $(\,\cdot\,|\,\cdot\,)$. 
Then, $\hat R\in\Psi_{n+1}(X)$ with $\sigma^0_{\hat R,-}(x,x)\neq0$, for every $x\in X$, and if $q=n_-=n+$,  $\sigma^0_{\hat R,+}(x,x)\neq0$, for every $x\in X$.
Let $\hat H\in\Psi_{n-1}(X)$ with 
\begin{equation}\label{e-gue210706ycdh}
\begin{split}
&S^{(q)}\hat H=\hat H=\hat HS^{(q)},\\
&\hat R\hat H\equiv \hat R\hat H\equiv S^{(q)}.
\end{split}
\end{equation}
Note that $\hat H$ is uniquely determined by \eqref{e-gue210706ycdh} (up to some smoothing operators). 
We can repeat the proof of Theorem~\ref{t-gue210724yyda} and deduce 


\begin{thm}\label{t-gue210706ycdg}
Recall that we work with the assumption that $q=n_-$. 
Let $f, g\in\mathcal{C}^\infty(X)$. We have 
\begin{equation}\label{e-gue210706ycds}
T_{f,-}\circ T_{g,-}-\sum^N_{j=0}\hat H^jT_{\hat C_{j,-}(f,g),-}\in\Psi_{n-N-1}(X),
\end{equation}
for every $N\in\mathbb N_0$, where $\hat C_{j,-}(f,g)\in\mathcal{C}^\infty(X)$, $\hat C_{j,-}$ is a universal bidifferential operator of order $\leq 2j$, $j=0,1,\ldots$, and 
\begin{equation}\label{e-gue210706ycdt}
\begin{split}
\hat C_{0,-}(f,g)=fg,\\
\hat C_{1,-}(f,g)-\hat C_{1,-}(g,f)=\imath\{f,g\}.
\end{split}
\end{equation}

If $q=n_-=n_+$, we have 
\begin{equation}\label{e-gue210706ycdy}
T_{f,+}\circ T_{g,+}-\sum^N_{j=0}\hat H^jT_{\hat C_{j,+}(f,g),+}\in\Psi_{n-N-1}(X),
\end{equation}
for every $N\in\mathbb N_0$, where $\hat C_{j,+}(f,g)\in\mathcal{C}^\infty(X)$, $\hat C_{j,+}$ is a universal bidifferential operator of order $\leq 2j$, $j=0,1,\ldots$, and 
\begin{equation}\label{e-gue210706ycdp}
\begin{split}
\hat C_{0,+}(f,g)=fg,\\
\hat C_{1,+}(f,g)-\hat C_{1,+}(g,f)=\imath\{f,g\}.
\end{split}
\end{equation}
\end{thm}




For $f, g\in\mathcal{C}^\infty(X)$, let 
\[f*g= \sum_{j=0}^{+\infty}\hat C_{j,\mp}(f,g) \nu^{-j}. \]
In general, $*$ is not associative. Whet $\hat R$ commutes with all $T_f$,  we can show that $*$ is associative. 
Assume that $X$ admits a transversal CR $\mathbb R$-action $\eta$. Take $R$ so that $R$ is induced by the $\mathbb R$-action. Suppose that $X$ admits a $\mathbb R$-invariant Hermitian metric $\langle\,\cdot\,|\,\cdot\,\rangle$ and let $(\,\cdot\,|\,\cdot\,)$ be the $L^2$ inner product for $\Omega^{0,q}(X)$ induced by $\langle\,\cdot\,|\,\cdot\,\rangle$. Put 
\begin{equation}\label{e-gue210725yyd}
\mathcal{C}^\infty(X)^{\mathbb R}:=\set{f\in\mathcal{C}^\infty(X);\,\eta^*f=f,\ \ \mbox{for every $\eta\in\mathbb R$}}.
\end{equation}
We can check that 
\begin{equation}\label{e-gue210725yydI}
\hat HT_{f,\mp}\equiv T_{f,\mp}\hat H,\ \ \mbox{for all $f\in\mathcal{C}^\infty(X)^{\mathbb R}$}. 
\end{equation}
Moreover, it is straightforward to see that 
\begin{equation}\label{e-gue210725yyda}
\mbox{$\hat C_{j,\mp}(f,g)\in\mathcal{C}^\infty(X)^{\mathbb R}$, for every $j=0,1,\ldots$, for all $f, g\in\mathcal{C}^\infty(X)^{\mathbb R}$}.
\end{equation}

\begin{thm}\label{t-gue210725yyd}
For $f, g\in\mathcal{C}^\infty(X)^{\mathbb R}$, let 
\[f*g= \sum_{j=0}^{+\infty}\hat C_{j,\mp}(f,g) \nu^{-j}. \]
The star product
 \[f*g= \sum_{j=0}^{+\infty}\hat C_{j,\mp}(f,g) \nu^{-j} ,\]
 $f, g\in \mathcal{C}^\infty(X)^{\mathbb R}$, is  associative. 
\end{thm}

\begin{proof}
Let $f, g, h\in\mathcal{C}^\infty(X)^{\mathbb R}$. 
Assume that $q\neq n_+$. For $q=n_-=n_+$, the proof is the same. 
Notice the associativity is equivalent to the equality
	\[ \sum_{l=0}^k\hat C_{l,-}(f,\,\hat C_{k-l,-}(g,h))=:\alpha_k=\beta_k:= \sum_{l=0}^k\hat C_{l,-}(\hat C_{k-l,-}(f,g),\,h) \]
	which is equivalent to
	\begin{equation}\label{eq:associativity} T_{\alpha_k,-} =T_{\beta_k,-} + A \end{equation}
	where $A\in\Psi_{n-1}(X)$. Now, we prove \eqref{eq:associativity} by induction over $k$. The case $k=0$ is easy, suppose that \eqref{eq:associativity} holds up to $k-1$, we shall prove \eqref{eq:associativity}. Now, for every $l=0,1,\dots,k$, from Theorem~\ref{t-gue210706ycdg} and the definition of $\hat C_{l,-}(\cdot,\,\cdot)$ in the previous paragraph implies that  
	\begin{equation*}  
	T_{\hat C_{l,-}(f,\,\hat C_{k-l,-}(g,h)),-}= \hat R^l(T_{f,-} T_{\hat C_{k-l,-}(g,h),-})- \sum_{s=0}^{l-1}\hat R^{l-s}T_{\hat C_s(f,{\hat C_{k-l}(g,h),-} ),-} +A_1 \,, \end{equation*}
	where $A_1\in\Psi_{n-1}(X)$ and similarly,	
	\begin{equation*}  T_{\hat C_{l,-}({\hat C_{k-l,-}(f,g),h}),-}=\hat R^l(T_{\hat C_{k-l,-}(f,g),-} T_{h,-})- \sum_{s=0}^{l-1}\hat R^{l-s} T_{\hat C_{s,-}({\hat C_{k-l,-}(f,g),h} ),-} +A_2 \,, \end{equation*}
	for an operator $A_2\in\Psi_{n-1}(X)$. Thus, summing over $l$ for a suitable $A_3\in\Psi_{n-1}(X)$, we have
	\begin{align*} 
	T_{\alpha_k-\beta_k,-} &= \sum_{l=0}^k\hat R^l ( T_{f,-}T_{\hat C_{k-l,-}(g,h),-} ) - \sum_{l=0}^k\sum_{j=0}^{l-1}\hat R^{l-j}T_{\hat C_{j,-}(f,{\hat C_{k-l,-}(g,h)} ),-} \\
	&- \sum_{l=0}^k\hat R^l ( T_{\hat C_{k-l,-}(f,g),-} T_{h,-} ) + \sum_{l=0}^k\sum_{j=0}^{l-1}\hat R^{l-j} T_{\hat C_{j,-}({\hat C_{k-l,-}(f,g),h} ),-}+A_3\,.
	 \end{align*}
	 The second summand can also be written
	 \[- \sum_{r=1}^k\hat R^r \sum_{l=r}^{k} T_{\hat C_{l-r,-}(f,{\hat C_{k-l,-}(g,h)} ),-}= - \sum_{r=1}^k\hat R^rT_{\sum_{s=0}^{k-r}\hat C_{s,-}(f,\hat C_{k-r-s,-}(g,h)),-} \]
	 and similarly for the forth summand, and then we can apply the inductive hypothesis; we get 
	 \begin{align} \label{eq:talphabeta}T_{\alpha_k-\beta_k,-} = \sum_{l=0}^k\hat R^l ( T_{f,-}T_{\hat C_{k-l,-}(g,h),-} )  - \sum_{l=0}^k\hat R^l ( T_{\hat C_{k-l,-}(f,g),-} T_{h,-} )+A_4\,  \end{align}
	 for a certain $A_4\in\Psi_{n-1}(X)$. By splitting the first sum into $l=0$ and $l\geq 1$, using for $l=0$ \eqref{e-gue210706ycds} and the fact that $T_{f,-}$ commutes with $\hat R$, we get
	 \begin{align*}
	 &\left(T_{f,-}\,\hat R^k T_{g,-}\,T_{h,-}-\sum_{s=0}^{k-1}T_{f,-}\hat R^{k-s}T_{\hat C_s(g,h),-} +A_5 \right)+\sum_{l=1}^{k}\hat R^l\,T_{f,-}\,T_{\hat C_{k-l,-}(g,h),-} \\
	 & =\hat R^k  T_{f,-}\,(T_{g,-}\,T_{h,-})+A_5\,,
	 \end{align*}
	 where again $A_5\in\Psi_{n-1}(X)$. Now, the second term on the right hand side of \eqref{eq:talphabeta} can be studied similarly and we get the theorem. 
	 \end{proof}

\section{$G$-invariant operators for locally free actions}\label{s-gue210812yyd}

Assume that $X$ admits a compact connected Lie group action $G$. In this work, we assume that 

\begin{ass}\label{a-gue170123I}
The $G$ action preserves $\omega_0$ and $J$. That is, $g^\ast\omega_0=\omega_0$ on $X$ and $g_\ast J=Jg_\ast$ on $HX$, for every $g\in G$, where $g^*$ and $g_*$ denote  the pull-back map and push-forward map of $G$, respectively. 
\end{ass}

Let $\mathfrak{g}$ denote the Lie algebra of $G$. For any $\xi \in \mathfrak{g}$, we write $\xi_X$ to denote the vector field on $X$ induced by $\xi$. That is, $(\xi_X u)(x)=\frac{\partial}{\partial t}\left(u(\exp(t\xi)\circ x)\right)|_{t=0}$, for any $u\in C^\infty(X)$.

\begin{defn}\label{d-gue170124}
The moment map associated to the form $\omega_0$ is the map $\mu:X \to \mathfrak{g}^*$ such that, for all $x \in X$ and $\xi \in \mathfrak{g}$, we have 
\begin{equation}\label{E:cmpm}
\langle \mu(x), \xi \rangle = \omega_0(\xi_X(x)).
\end{equation}
\end{defn}

Let $b$ be the nondegenerate bilinear form on $HX$ such that 
\begin{equation}\label{E:biform}
b(\cdot , \cdot) = d\omega_0(\cdot , J\cdot).
\end{equation}
In this work, we assume that 

\begin{ass}\label{a-gue170123II}
$0$ is a regular value of $\mu$, the action $G$ on $\mu^{-1}(0)$ is locally free and 
\begin{equation}\label{e-gue200120yydI}
\mbox{$\underline{\mathfrak{g}}_x\bigcap \underline{\mathfrak{g}}^{\perp_b}_x=\set{0}$ at every point $x\in Y$}, 
\end{equation}
where  $\underline{\mathfrak{g}}={\rm Span\,}(\xi_X;\, \xi\in\mathfrak{g})$, 
$\underline{\mathfrak{g}}^{\perp_b}=\set{v\in HX;\, b(\xi_X,v)=0,\ \ \forall \xi_X\in\underline{\mathfrak{g}}}$. 
\end{ass}

By Assumption~\ref{a-gue170123II}, $\mu^{-1}(0)$ is a $d$-codimensional submanifold of $X$. Let $Y:=\mu^{-1}(0)$ and let $HY:=HX\bigcap TY$. 
Note that if the Levi form is positive at $Y$, then \eqref{e-gue200120yydI} holds. 
Fix a $G$-invariant smooth Hermitian metric $\langle\, \cdot \,|\, \cdot \,\rangle$ on $\mathbb{C}TX$ so that $T^{1,0}X$ is orthogonal to $T^{0,1}X$, $\underline{\mathfrak{g}}$ is orthogonal to $HY\bigcap JHY$ at every point of $Y$, $\langle \, u \,|\, v \, \rangle$ is real if $u, v$ are real tangent vectors, $\langle\,R\,|\,R\,\rangle=1$. 
Fix $g\in G$. Let $g^*:\Lambda^r_x(\Complex T^*X)\To\Lambda^r_{g^{-1}\circ x}(\Complex T^*X)$ be the pull-back map. Since $G$ preserves $J$, we have $g^*:T^{*0,q}_xX\To T^{*0,q}_{g^{-1}\circ x}X$, for all $x\in X.$ Thus, for $u\in\Omega^{0,q}(X)$, we have $g^*u\in\Omega^{0,q}(X)$. Put $\Omega^{0,q}(X)^G:=\set{u\in\Omega^{0,q}(X);\, g^*u=u,\ \ \forall g\in G}.$
Since the Hermitian metric $\langle\,\cdot\,|\,\cdot\,\rangle$ on $\Complex TX$ is $G$-invariant, the $L^2$ inner product $(\,\cdot\,|\,\cdot\,)$ on $\Omega^{0,q}(X)$ 
induced by $\langle\,\cdot\,|\,\cdot\,\rangle$ is $G$-invariant. Let $u\in L^2_{(0,q)}(X)$ and $g\in G$, we can also define $g^*u$ in the standard way. Put 
\[L^2_{(0,q)}(X)^G:=\set{u\in L^2_{(0,q)}(X);\, g^*u=u,\ \ \forall g\in G}.\]
Put 
\[({\rm Ker\,}\Box^q_b)^G:={\rm Ker\,}\Box^q_b\bigcap L^2_{(0,q)}(X)^G.\] The $G$-invariant Szeg\H{o} projection is the orthogonal projection 
\[S^{(q)}_G:L^2_{(0,q)}(X)\To ({\rm Ker\,}\Box^q_b)^G\]
with respect to $(\,\cdot\,|\,\cdot\,)$. Let $S^{(q)}_G(x,y)\in\mathcal{D}'(X\times X,T^{*0,q}X\boxtimes(T^{*0,q}X)^*)$ be the distribution kernel of $S^{(q)}_G$.

As before, we will assume that $\Box^q_b$ has $L^2$ closed range. Let 
\[Q_G: L^2_{(0,q)}(X)\To L^2_{(0,q)}(X)^G\]
be the orthogonal projection with respect to $(\,\cdot\,|\,\cdot\,)$. Let
\begin{equation}\label{e-gue210729yyd}
\begin{split}
&S^G_-:=Q_G\circ S_-: L^2_{(0,q)}(X)\To ({\rm Ker\,}\Box^q_b)^G,\\
&S^G_+:=Q_G\circ S_+: L^2_{(0,q)}(X)\To ({\rm Ker\,}\Box^q_b)^G,
\end{split}
\end{equation}
where $S_-$, $S_+$ are as in Theorem~\ref{t-gue161109I}. We have $S^{(q)}_G=S^G_-+S^G_+$. If $q\neq n_+$, then $S^G_+\equiv0$. Recall that we work with the assumption that $q=n_-$. 

\begin{defn}\label{d-gue210729yyd}
Let $P\in L^{\ell}_{{\rm cl\,}}(X,T^{*0,q}X\boxtimes(T^{*0,q}X)^*)$, $\ell\in\mathbb Z$, $\ell\leq0$. We say that 
$P$ is in $L^{\ell}_{{\rm cl\,}}(X,T^{*0,q}X\boxtimes(T^{*0,q}X)^*)^G$ if $g^*(Pu)=P(g^*u)$, for every $u\in L^2_{(0,q)}(X)$ and every $g\in G$. 
\end{defn}

\begin{defn}\label{d-gue210801yyd}
Let $P\in L^\ell_{{\rm cl\,}}(X,T^{*0,q}X\boxtimes(T^{*0,q}X)^*)^G$ with scalar principal symbol, $\ell\leq0$, $\ell\in\mathbb Z$.
The $G$-invariant Toeplitz operator is given by 
\begin{equation}\label{e-gue210707ycdz}
\begin{split}
&T^G_P:=S^{(q)}_G\circ P\circ S^{(q)}_G: L^2_{(0,q)}(X)\To({\rm Ker\,}\Box^q_b)^G,\\
&T^G_{P,-}:=S^G_-\circ P\circ S^G_-: L^2_{(0,q)}(X)\To({\rm Ker\,}\Box^q_b)^G,\\
&T^G_{P,+}:=S^G_+\circ P\circ S^G_+: L^2_{(0,q)}(X)\To({\rm Ker\,}\Box^q_b)^G,
\end{split}
\end{equation}
where $S_-, S_+$ are as in \eqref{e-gue210707yyd}. 
If $q\neq n_+$, then $T^G_{P,+}\equiv0$ on $X$. 

Let $f\in\mathcal{C}^\infty(X)^G$, we write $T^G_f:=T^G_{M_f}$,  $T^G_{f,\mp}:=T^G_{M_f,\mp}$. 
\end{defn}

For every $x\in X$, put $G_x:=\set{g\in G;\, gx=x}$ and let $|G_x|$ denote the cardinal number of $G_x$. Fix $x\in\mu^{-1}(0)$, consider the linear map 
\[
\renewcommand{\arraystretch}{1.2}
\begin{array}{rll}
R_x:\underline{\mathfrak{g}}_x&\To&\underline{\mathfrak{g}}_x,\\
u&\To& R_xu,\ \ \langle\,R_xu\,|\,v\,\rangle=\langle\,d\omega_0(x)\,,\,Ju\wedge v\,\rangle.
\end{array}
\]
Let $\det R_x=\lambda_1(x)\cdots\lambda_d(x)$, where $\lambda_j(x)$, $j=1,2,\ldots,d$, are the eigenvalues of $R_x$. Fix $x\in\mu^{-1}(0)$, put $Y_x=\set{g\circ x;\, g\in G}$. $Y_x$ is a $d$-dimensional submanifold of $X$. The $G$-invariant Hermitian metric $\langle\,\cdot\,|\,\cdot\,\rangle$ induces a volume form $dv_{Y_x}$ on $Y_x$. Put 
\begin{equation}\label{e-gue170108em}
V_{{\rm eff\,}}(x):=\int_{Y_x}dv_{Y_x}.
\end{equation}

The following theorem generalizes the results  in \cite{hsiaohuang} to Toeplitz operators and to the action of $G$ on $X$ is locally free. 

\begin{thm} \label{thm:toeplitz}
Recall that we work with the assumption that $q=n_-$. 
With the notations and assumptions used above, let $P\in L^\ell_{{\rm cl\,}}(X,T^{*0,q}X\boxtimes(T^{*0,q}X)^*)^G$ with scalar principal symbol, $\ell\leq0$, $\ell\in\mathbb Z$.
Let $D$ be an open set in $X$ such that the intersection $\mu^{-1}(0)\cap D= \emptyset$. Then $T^G_{P,\mp}\equiv 0$ on $D$.
	
	Let $p\in \mu^{-1}(0)$ and let $G_p=\set{g_1,\ldots,g_r}$, $r=|G_p|$. Let $U$ a local neighborhood of $p$ with local coordinates $(x_1,\dots\,x_{2n+1})$. Then, the distributional kernel of $T_{P,-}^{G}$ satisfies
\begin{equation}\label{e-gue210802yyd}
T_{P,-}^{G}(x,y)\equiv\sum^{|G_p|}_{j=1}\int_0^{\infty} e^{\imath t\,\Phi_-( x,g_jy)}a_{P,-}( x,\,g_jy,\,t)\,\mathrm{d}t\ \ \ \mbox{on $U\times U$}.
\end{equation}
The phase $\Phi_-$ is described in Section \S\ref{sec:pre} and it is equal to the phase of $S_{-}^{G}(x,\,y)$ in~\cite{hsiaohuang}, the symbol 
$a_{P,-}$ satisfies the following properties
	\[a_{P,-}(x,\,y,\,t)\sim \sum_{j=0}^{+\infty} a_{P,-}^{j}(x,\,y)\,t^{n-d/2-j}\]
	in $S^{n-d/2}_{1,0}(U\times U\times \mathbb{R},\,T^{*\,0,q}X\boxtimes(T^{*\,0,q}X)^*)$, 
	\[a_{P,-}^{j}(x,\,y)\in\mathcal{C}^{\infty}(U\times U,T^{*0,q}X\boxtimes(T^{*0,q}X)^*),\ \ j\in\mathbb N_0,\]
	and for every $x\in\mu^{-1}(0)$, 
	\begin{equation}\label{e-gue210802yydI}
	a_{P,-}^{0}(x,\,x)=2^{d-1}\frac{1}{V_{{\rm eff\,}}(x)|G_x|}\pi^{-n-1+\frac{d}{2}}\abs{\det R_x^{-\frac{1}{2}}}\abs{\det\mathcal{L}_{x}}\sigma^0_P(x,-\omega_0(x))\tau_{x,n_-}.	\end{equation}
	
	Moreover, for every fix $y\in U$, we consider $T_{P,-}^{G}(x,y)$ as a distribution in $x$ variables, then
	\begin{equation}\label{e-gue210802yydII}
	T_{P,-}^{G}(x,\,y)\equiv\vert G_y\rvert\,\int_0^{\infty} e^{\imath t\,\Phi_-(x,y)}a_{P,-}(x,\,y,\,t)\,\mathrm{d}t\ \ \mbox{on $U$},\end{equation}
	for every $x,\,y\in U$, $\Phi_-$, $a_{P,-}$ are as above. 
			
	If $q\neq n_+$, then $T_{P,+}^{G}\equiv0$. The case $q=n_+=n_-$ can be stated similarly except that the phase of Toeplitz kernel $T_{P,+}^{G}(x,\,y)$ is given by $\Phi_+(x,\,y)$.
\end{thm}

\begin{defn}\label{d-gue210803yyd}
Let $H: \Omega^{0,q}(X)\To\Omega^{0,q}(X)$ be a continuous operator with distribution kernel $H(x,y)\in\mathcal{D}'(X\times X,T^{*0,q}X\boxtimes(T^{*0,q}X)^*)$. 
We say that $H$ is a complex Fourier integral operator of $G$-invariant Szeg\H{o} type of order $k\in\mathbb Z$ if $H$ is smoothing away $\mu^{-1}(0)$ and let $p\in \mu^{-1}(0)$ and let $G_p=\set{g_1,\ldots,g_r}$, $r=|G_p|$. Let $D$ a local neighborhood of $p$ with local coordinates $(x_1,\dots\,x_{2n+1})$. Then, the distributional kernel of $H$ satisfies
\[\begin{split}
&H(x,y)\equiv H_-(x,y)+H_+(x,y)\ \ \mbox{on $D$},\\
&\mbox{$H_-(x,y)\equiv\sum^{|G_p|}_{j=1}\int^\infty_0 e^{\imath t\Phi_-(x,g_jy)}a_-(x,g_jy,t)dt$ on $D$},\\
&\mbox{$H_+(x,y)\equiv\sum^{|G_p|}_{j=1}\int^\infty_0e^{\imath t\Phi_{+}(x,g_jy)t}a_+(x, g_jy, t)dt$ on $D$},\end{split}\]
where $a_-, a_+\in S^{k-\frac{d}{2}}_{{\rm cl\,}}(D\times D\times\mathbb{R}_+,T^{*0,q}X\boxtimes(T^{*0,q}X)^*)$, $a_+=0$ if $q\neq n_+$, where $\Phi_-$, $\Phi_+$ are as in 
Theorem~\ref{thm:toeplitz}. We write $\sigma^0_{H,-}(x,y)$ to denote the leading term of the expansion \eqref{e-gue161110r} of $a_{-}(x,y,t)$. If $q=n_+$, we write $\sigma^0_{H,+}(x,y)$ to denote the leading term of the expansion \eqref{e-gue161110r} of $a_{+}(x,y,t)$. Note that $\sigma^0_{H,-}(x,y)$ and $\sigma^0_{H,+}(x,y)$ depend on the choices of the phases $\varphi_-$ and $\varphi_+$ but $\sigma^0_{H,-}(x,x)$ and $\sigma^0_{H,+}(x,x)$ are independent of the choices of the phases $\varphi_-$ and $\varphi_+$.

Let $\Psi_k(X)^G$ denote the space of all complex Fourier integral operators of Szeg\H{o} type of order $k$. 
\end{defn}

\begin{thm} \label{thm:composition}
With the notations and assumptions used above, let $P\in L^\ell_{{\rm cl\,}}(X,T^{*0,q}X\boxtimes(T^{*0,q}X)^*)^G$ with scalar principal symbol, $Q\in L^k_{{\rm cl\,}}(X,T^{*0,q}X\boxtimes(T^{*0,q}X)^*)^G$ with scalar principal symbol, $\ell, k\leq0$, $\ell, k\in\mathbb Z$. We have 
\begin{equation}\label{e-gue210803yyd}\begin{split}
&[T^G_{P,-}, T^G_{Q,-}]-T^G_{{\rm Op\,}(\imath\{\sigma^0_P,\sigma^0_Q\}_-)} \in\Psi_{n-2+\ell+k}(X)^G,\\
&[T^G_{P,+}, T^G_{Q,+}]-T^G_{{\rm Op\,}(\imath\{\sigma^0_P,\sigma^0_Q\}_+)} \in\Psi_{n-2+\ell+k}(X)^G.
\end{split}\end{equation}
\end{thm}

For every $j\in\mathbb Z$, $j\leq0$, put 
\begin{equation}\label{e-gue210803yydI}
\hat S^j_G\:=\set{f(x)(\sigma^0_{\triangle_b})^{\frac{j}{2}}\in\mathcal{C}^\infty(TX);\, f(x)\in\mathcal{C}^\infty(X)^G}
\end{equation}
and let 
\begin{equation}\label{e-gue210803yydII}
\hat S_G:=\bigcup_{j\in\mathbb Z, j\leq0}\hat S_j.
\end{equation}

 For $a, b\in\hat S_G$, define 
\begin{equation}\label{e-gue210803yydIII}
a\hat{+}b:=\mbox{leading term of $a+b$}. 
\end{equation}
To simplify the notations, for $a\in\hat S_G$, we denote 
\begin{equation}\label{e-gue210803ycd}
T^G_a:=T^G_{{\rm Op\,}(a)},\ \ T^G_{a,-}:=T^G_{{\rm Op\,}(a),-},\ \ T^G_{a,+}:=T_{{\rm Op\,}(a),+}.
\end{equation}

We can now establish star product for $\hat S^j_G$.

\begin{thm}\label{t-gue210803ycd}
Recall that we work with the assumption that $q=n_-$. 
Let $a\in\hat S^\ell_G$, $b\in\hat S^k_G$, $\ell, k\in\mathbb Z$, $\ell, k\leq0$. We have
\begin{equation}\label{e-gue210803ycdI}
T^G_{a,-}\circ T^G_{b,-}-\sum^N_{j=0}T^G_{C_{j,-}(a,b),-}\in\Psi_{n-N-1+\ell+k}(X)_G,
\end{equation}
for every $N\in\mathbb N_0$, where $C_{j,-}(a,b)\in\hat S^{\ell+k-j}_G$, $C_{j,-}$ is a universal bidifferential operator of order $\leq 2j$, $j=0,1,\ldots$, and 
\begin{equation}\label{e-gue210803ycdII}
\begin{split}
C_{0,-}(a,b)=ab,\\
C_{1,-}(a,b)-C_{1,-}(b,a)=\imath\{a,b\}_-.
\end{split}
\end{equation}

If $q=n_-=n_+$, we have
\begin{equation}\label{e-gue210803ycdIII}
T^G_{a,+}\circ T^G_{b,+}-\sum^N_{j=0}T^G_{C_{j,+}(a,b),+}\in\Psi_{n-N-1+\ell+k}(X)_G,
\end{equation}
for every $N\in\mathbb N_0$, where $C_{j,+}(a,b)\in\hat S^{\ell+k-j}_G$, $C_{j,+}$ is a universal bidifferential operator of order $\leq 2j$, $j=0,1,\ldots$, and 
\begin{equation}\label{e-gue210803ycdx}
\begin{split}
C_{0,+}(a,b)=ab,\\
C_{1,+}(a,b)-C_{1,+}(b,a)=\imath\{a,b\}_{+}.
\end{split}
\end{equation} 

Moreover, the star product
 \[a*b= \sum_{j=0}^{+\infty} C_{j,\mp}(a,b) \nu^{-j} ,\]
 $a, b\in\hat S_G$, is associative.
\end{thm}

\subsection{Fourier component of $G$-invariant operators in the presence of a locally free circle action}

Before starting the analysis of $G$-invariant Szeg\H{o} operator in the presence of a locally free circle action, we shall expose some historical remarks. When $X$ is the circle bundle of a quantizable K\"ahler manifold R. Paoletti \cite{pao} investigated the asymptotics of the $G$-invariant Szeg\H{o} kernel by adapting Heisenberg local coordinates for locally free actions; it should be mentioned that R. Paoletti studied equivariant kernel of projectors onto any isotype. From the perspective of the Bergman kernel, the analysis of $G$-invariant Toeplitz operators algebra was carried out by \cite{mz} using analytic localization techniques of Bismut and Lebeau for spin\textsuperscript{c} Dirac operators, X. Ma and W. Zhang studied among other things the commutator of $G$-invariant Toeplitz operators and their relations with symplectic reduction. 

By the motivation explained in the previous paragraph we now assume that $X$ admits a CR and transversal $S^1$-action which is locally free on $\mu^{-1}(0)$, $T\in \mathcal{C}^{\infty}(X,\,TX)$ denotes the global real vector field given by the infinitesimal circle action. We will take $T$ to be our Reeb vector field $R$. 
We assume that 

\begin{ass}\label{a-gue170128}
\[
\mbox{$T$ is transversal to the space $\underline{\mathfrak{g}}$ at every point $p\in\mu^{-1}(0)$},
\]
\begin{equation}\label{e-gue170111ryII}
e^{i\theta}\circ g\circ x=g\circ e^{i\theta}\circ x,\  \ 
\mbox{for every $x\in X$, $\theta\in[0,2\pi[$, $g\in G$}, 
\end{equation}
and 
\[
\mbox{$G\times S^1$ acts locally free near $\mu^{-1}(0)$}. 
\]
\end{ass}

Let us state the last piece of notation. Since the action of $S^1$ is locally free on $\mu^{-1}(0)$, then $\mu^{-1}(0)/S^1$ is an orbifold, let us denote with $$\pi\,:\,\mu^{-1}(0)\rightarrow \mu^{-1}(0)/S^1$$ the projection. Furthermore the action of $S^1$ commutes with the one of $G$, then we have an action of $G$ on $\mu^{-1}(0)/S^1$. Notice that the action of $G$ is locally free on $\mu^{-1}(0)/S^1$, in fact by the transversality assumption we have $$\mathbb{C}T_xX=\mathbb{C}T(x)\oplus \mathbb{C}T_x^{1,0}X\oplus \mathbb{C}T_x^{0,1}X \qquad(x\in X)$$ and also $\xi_X(x)\in \mathrm{Ker}(\omega_0)$ for each $\xi\in\mathfrak{g}$ whenever $x\in\mu^{-1}(0)$.
Given $x\in \mu^{-1}(0)$ and $g\in G_{\pi(x)}$, there exist $\vert S_x^1 \rvert$ elements $e^{\imath\,\theta_{g,j}}\in S^1$ such that 
\[g\circ x= e^{-\imath\,\theta_{g,j}}\circ x\,. \]

Let $u\in\Omega^{0,q}(X)$ be arbitrary. Define
\[
Tu:=\frac{\pr}{\pr\theta}\bigr((e^{i\theta})^*u\bigr)|_{\theta=0}\in\Omega^{0,q}(X).
\]
For every $m\in\mathbb Z$, let
\[
\renewcommand{\arraystretch}{1.2}
\begin{array}{ll}
&\Omega^{0,q}_m(X):=\set{u\in\Omega^{0,q}(X);\, Tu=imu},\ \ q=0,1,2,\ldots,n,\\
&\Omega^{0,q}_{m}(X)^G=\set{u\in\Omega^{0,q}(X)^G;\, Tu=imu},\ \ q=0,1,2,\ldots,n.
\end{array}
\]
We denote $\mathcal{C}^\infty_m(X):=\Omega^{0,0}_m(X)$, $C^\infty_m(X)^G:=\Omega^{0,0}_m(X)^G$. We now assume that the Hermitian metric $\langle\,\cdot\,|\,\cdot\,\rangle$ on $\Complex TX$ is $G\times S^1$ invariant.  Then the $L^2$ inner product $(\,\cdot\,|\,\cdot\,)$ on $\Omega^{0,q}(X)$ 
induced by $\langle\,\cdot\,|\,\cdot\,\rangle$ is $G\times S^1$-invariant. We then have 
\[
\renewcommand{\arraystretch}{1.2} 
\begin{array}{ll}
&Tg^*\ol{\pr}^*_b=g^*T\ol{\pr}^*_b=\ol{\pr}^*_bg^*T=\ol{\pr}^*_bTg^*\ \ 
\mbox{on $\Omega^{0,q}(X)$},\ \ \forall g\in G,\\
&Tg^*\Box^q_b=g^*T\Box^q_b=\Box^q_bg^*T=\Box^q_bTg^*\ \ \mbox{on $\Omega^{0,q}(X)$},\ \
\forall g\in G,
\end{array}
\]
where $\ol{\pr}^*_b$ is the $L^2$ adjoint of $\ddbar_b$ with respect to $(\,\cdot\,|\,\cdot\,)$. 

Let $L^2_{(0,q), m}(X)^G$ be
the completion of $\Omega_m^{0,q}(X)^G$ with respect to $(\,\cdot\,|\,\cdot\,)$. 
The $m$-th $G$-invariant Szeg\H{o} projection is the orthogonal projection 
$S^{(q)}_{G,m}:L^2_{(0,q)}(X)\To ({\rm Ker\,}\Box^q_b)^G_m$
with respect to $(\,\cdot\,|\,\cdot\,)$. For $f\in\mathcal{C}^\infty(X)^G_0$, the $m$-th $G$-invariant Toeplitz operator is given by 
\begin{equation}\label{e-gue210803yyda}
T^G_{f,m}:=S^{(q)}_{G,m}\circ M_f\circ S^{(q)}_{G,m}: L^2_{(0,q)}(X)\To L^2_{(0,q),m}(X)^G.
\end{equation}

The following is a generalization of Theorem $1.8$ in \cite{hsiaohuang} when the action of $G$ on $X$ is locally free.

\begin{thm} \label{thm:kfourierszego}
	Let $(X,\,T^{1,0}X)$ be a compact CR manifold with a locally free action of a compact connected Lie group $G$ and transversal free action of $S^1$ satisfying the assumptions above. $f\in\mathcal{C}^\infty(X)^G_0$.If $q\neq n_-$, then $T^G_{f,m}\equiv O(m^{-\infty})$ on $X$.
	
	Suppose $q= n_-$ and let $D$ be an open set in $X$ such that the intersection $\mu^{-1}(0)\cap D= \emptyset$. Then $S^{(q)}_{G,m}\equiv O(m^{-\infty})$ on $D$.
	
	Let $p\in \mu^{-1}(0)$ and let $U$ a local neighborhood of $p$ with local coordinates $(x_1,\dots\,x_{2n+1})$. Then, if $q= n_-$, for every fix $y\in U$, we consider $T_{f,m}^{G}(x,y)$ as a $m$-dependent smooth function in $x$, then 
	\[T^G_{f,m}(x,\,y)=\sum_{h\in G_{\pi(y)}}\sum_{j=1}^{\lvert S_x^1 \rvert}e^{\imath m\,\theta_{h,j}} e^{\imath m \,\Psi(x,\,y)}\,b(x,\,y,\,m)+O(m^{-\infty})\,. \]
	for every $x\in U$. The phase $\Psi$ is the same phase function obtained for the free case (see Section~\ref{s-gue210804yyd}), the symbol satisfies
	\[b(x,\,y,\,m)\in S^{n-d/2}_{\mathrm{loc}}(1,\,U\times U,\,T^{*\,(0,q)}X\boxtimes(T^{*\,(0,q)})^*)\, \]
	and the leading term of $b(x,y,m)$ along the diagonal at $x\in\mu^{-1}(0)$ is given by 
		\[b_0(x,\,x):= 2^{d-1}\frac{1}{V_{{\rm eff\,}}(x)|G_x|}\pi^{-n-1+\frac{d}{2}}\abs{\det R_x}^{-\frac{1}{2}}\abs{\det\mathcal{L}_{x}}\tau_{x,n_-}. \]
		
\end{thm}

We can study the commutator and star product.

\begin{thm} \label{thm:compositionFouriercomponents}
	With the same assumptions as above, let $f, g\in\mathcal{C}^\infty(X)^G_0$. Let $q=n_-$. Then, as $m\gg1$, 
	\begin{equation}\label{e-gue210803ycds}
	\norm{T^G_{f,m}\circ T^G_{g,m}-T^G_{g,m}\circ T^G_{f,m}-\frac{1}{m}T^G_{\imath\{f,g\},m}}=O(m^{-2}),\end{equation}
	and 
\begin{equation}\label{e-gue210803ycdt}
\norm{T^G_{f,m}\circ T^G_{g,m}-\sum^N_{j=0}m^{-j}T^G_{C_j(f,g),m}}=O(m^{-N-1})
\end{equation}
in $L^2$  operator norm, for every $N\in\mathbb N$, where $C_j(f,g)\in\mathcal{C}^\infty(X)^G_0$, 
$C_j$ is a universal bidifferential operator of order $\leq 2j$, $j=0,1,\ldots$, and 
\begin{equation}\label{e-gue210803ycdu}
\begin{split}
C_{0}(f,g)=fg,\\
C_{1}(f,g)-C_{1}(g,f)=\imath\{f,g\}.
\end{split}
\end{equation}

 Moreover, the star product
 \begin{equation}\label{e-gue210805yyd}
 f*g= \sum_{j=0}^{+\infty} C_{j}(f,g) \nu^{-j} ,
\end{equation}
 $f, g\in\mathcal{C}^\infty(X)^G_0$, is associative. 
\end{thm}

\section{Preliminaries on local coordinates}
\label{sec:pre}

We do make use of local coordinates defined in \cite{hsiaohuang}, which we briefly recall in this section. By Theorem 3.6 in \cite{hsiaohuang}, there exist local coordinates $v=(v_1,\dots,v_d)$ of $G$ in a small neighborhood $V_0$ of $e$ with $v(e)=(0,\,\dots,\,0)$, local coordinates $x=(x_1\,\dots,x_{2n+1})$ defined in a neighborhood $U_1\times U_2$ of $p$, where $U_1\subseteq \mathbb{R}^d$ (resp. $U_d\subseteq \mathbb{R}^{2n+1-d}$) is an open set of $0\in \mathbb{R}^d$ (resp. $0\in \mathbb{R}^{2n+1-d}$) and $p\equiv 0\in \mathbb{R}^{2n+1}$, and a smooth function $\gamma=(\gamma_1,\dots,\,\gamma_d)\in \mathcal{C}^{\infty}(U_2,U_1)$ with $\gamma(0)=0$ such that
\begin{align*} &(v_1,\dots,v_d)\circ (\gamma(x_{d+1},\dots,x_{2n+1}),x_{d+1},\dots,x_{2n+1}) \\
	&=(v_1+\gamma_1(x_{d+1},\dots,x_{2n+1}),\dots,\,v_d+\gamma_d(x_{d+1},\dots,x_{2n+1}),\,x_{d+1},\dots,x_{2n+1}) \end{align*}
for each $(v_1,\dots,\,v_d)\in V$ and $(x_{d+1},\dots,x_{2n+1})\in U_2$. Furthermore, we have
\[\mathfrak{g}=\mathrm{span}\left\{\partial_{{x}_{j}}\right\}_{j=1,\dots, d}\,,\quad \mu^{-1}(0)\cap U =\{x_{d+1} = \dots=x_{2d}=0 \}\,, \]
on $\mu^{-1}(0)\cap U$ there exist smooth functions $a_j$'s with $a_j(0)=0$ for every $0\leq j\leq d$ and independent on $x_1,\dots,x_{2d},\,x_{2n+1}$ such that
\[J\left(\partial_{{x}_{j}}\right)=\partial_{{x}_{d+j}}+a_j(x)\partial_{{x}_{2n+1}}\qquad j=1,\dots,d\,, \]
the Levi form $\mathcal{L}_p$, the Hermitian metric $\langle\,\cdot\,|\,\cdot\,\rangle$ and the $1$-form $\omega_0$ can be written 
\[\mathcal{L}_p(Z_j,\,\overline{Z}_k)=\mu_j\,\delta_{j,k},\qquad \langle Z_j\vert\,\overline{Z}_k \rangle=\delta_{j,k}\qquad (1\leq j,k\leq n)\,, \]
\begin{align}\label{e-gue161219m}
	\omega_0(x)=&(1+O(\lvert x\rvert))\mathrm{d}x_{2n+1}+\sum_{j=1}^d 4\mu_jx_{d+j}\mathrm{d}x_j +\sum_{j=d+1}^n 2\mu_jx_{2j}\mathrm{d}x_{2j-1} \\
	&-\sum_{j=d+1}^n 2\mu_jx_{2j-1}\mathrm{d}x_{2j}+\sum_{j=d+1}^{2n} b_jx_{2n+1}\mathrm{d}x_j+O(\lvert x\rvert^2)
\end{align}
where $b_{d+1},\dots,b_{2n}\in \mathbb{R}$,
\[T_p^{1,0}X=\mathrm{span}\{Z_1,\dots,Z_n\} \]
and
\begin{align*}Z_j&=\frac{1}{2}\,(\partial_{{x}_{j}}-\imath\, \partial_{{x}_{d+j}})(p)\qquad(j=1,\dots,d)\,,\\ Z_j&=\frac{1}{2}\,(\partial_{{x}_{2j-1}}-\imath\, \partial_{{x}_{2j}})(p)\qquad(j=d+1,\dots,n)\,.
\end{align*}

\subsection{The phase functions $\Phi_-(x,y)$ and $\Psi(x,y)$}\label{s-gue210804yyd}

The phase function $\Phi_-(x,y)\in\mathcal{C}^\infty(U\times U)$ is independent of $(x_1,\ldots,x_d)$ and $(y_1,\ldots,y_d)$. Hence, $\Phi_-(x,y)=\Phi_-((0,x''),(0,y'')):=\Phi_-(x'',y'')$. Moreover, there is a constant $c>0$ such that 
\begin{equation}\label{e-gue170126}
{\rm Im\,}\Phi_-(x'',y'')\geq c\Bigr(\abs{\hat x''}^2+\abs{\hat y''}^2+\abs{\mathring{x}''-\mathring{y}''}^2\Bigr),\ \ \forall ((0,x''),(0,y''))\in U\times U.
\end{equation}

Furthermore, 
\begin{equation}\label{eq:phase phi-}
\renewcommand{\arraystretch}{1.3}
\begin{array}{cl}
\Phi_-(x'', y'')&=-x_{2n+1}+y_{2n+1}+2i\sum^d_{j=1}\abs{\mu_j}y^2_{d+j}+2i\sum^d_{j=1}\abs{\mu_j}x^2_{d+j}\\
&+i\sum^{n}_{j=d+1}\abs{\mu_j}\abs{z_j-w_j}^2 +\sum^{n}_{j=d+1}i\mu_j(\ol z_jw_j-z_j\ol w_j)\\
&+\sum^d_{j=1}(-b_{d+j}x_{d+j}x_{2n+1}+b_{d+j}y_{d+j}y_{2n+1})\\
&+\sum^n_{j=d+1}\frac{1}{2}(b_{2j-1}-ib_{2j})(-z_jx_{2n+1}+w_jy_{2n+1})\\
&+\sum^n_{j=d+1}\frac{1}{2}(b_{2j-1}+ib_{2j})(-\ol z_jx_{2n+1}+\ol w_jy_{2n+1})\\
&+(x_{2n+1}-y_{2n+1})f(x, y) +O(\abs{(x, y)}^3),
\end{array}
\end{equation}
where $z_j=x_{2j-1}+ix_{2j}$, $w_j=y_{2j-1}+iy_{2j}$, $j=d+1,\ldots,n$, $\mu_j$, $j=1,\ldots,n$, and $b_{d+1}\in\Real,\ldots,b_{2n}\in\Real$ are as in \eqref{e-gue161219m} and $f$ is smooth and satisfies $f(0,0)=0$, $f(x, y)=\ol f(y, x)$.

We now assume that $X$ admits an $S^1$ action: $S^1\times X\rightarrow X$. We will use the same notations as before.   Let $p\in\mu^{-1}(0)$.  By using the proof of Theorem 3.6 in \cite{hsiaohuang}, there exist local coordinates $v=(v_1,\dots,v_d)$ of $G$ in a small neighborhood $V_0$ of $e$ with $v(e)=(0,\,\dots,\,0)$, local coordinates $x=(x_1\,\dots,x_{2n+1})$ defined in a neighborhood $U_1\times U_2$ of $p$, where $U_1\subseteq \mathbb{R}^d$ (resp. $U_d\subseteq \mathbb{R}^{2n+1-d}$) is an open set of $0\in \mathbb{R}^d$ (resp. $0\in \mathbb{R}^{2n+1-d}$) and $p\equiv 0\in \mathbb{R}^{2n+1}$, and a smooth function $\gamma=(\gamma_1,\dots,\,\gamma_d)\in\mathcal{C}^{\infty}(U_2,U_1)$ with $\gamma(0)=0$ such that $T=-\frac{\pr}{\pr x_{2n+1}}$ and all the properties for the local coordinates in  Section~\ref{sec:pre} hold.
The phase function $\Psi$ satisfies $\Psi(x,y)=-x_{2n+1}+y_{2n+1}+\hat\Psi(\mathring{x}'',\mathring{y}'')$, where $\hat\Psi(\mathring{x}'',\mathring{y}'')\in\mathcal{C}^\infty(U\times U)$, $\mathring{x}''=(x_{d+1},\ldots,x_{2n})$,  $\mathring{y}''=(y_{d+1},\ldots,y_{2n})$, and $\Psi$ satisfies \eqref{e-gue170126} and \eqref{eq:phase phi-}. 

\section{Proofs of Theorem~\ref{thm:toeplitz}, Theorem~\ref{thm:composition} and Theorem~\ref{t-gue210803ycd}}

We first prove Theorem~\ref{thm:toeplitz}.The first step in the proof consists in showing that the kernel $T^G_{P,-}(x,\,y)$ localizes on small neighborhoods of $g_l$'s in $G_y$. More precisely, we have
\[T^G_{P,-}(x,\,y)=\int_G T_{P,-}(x,\,g\circ y)\,\mathrm{dV}_G(g)\,. \]
($\mathrm{dV}_G(g)$ is the unique $G$-invariant measure such that $\lvert G\rvert_{\mathrm{dV}_G}=1$). In a similar way as in Section \S$3.4$ in \cite{hsiaohuang} one can prove that $T^G_{P,-}\equiv 0$ away from $\mu^{-1}(0)$; from now on $p\in \mu^{-1}(0)$ and $(x,y)$ are local coordinates defined in Section \S\ref{sec:pre} on $U\times U$ where $U$ is a  small of $p$ in $X$.
 
\begin{lem} Let $\chi_l$ be a bump function vanishing outside a small neighborhood $V_l$ of each $g_l\in G_y$ and $g_1$ is the identity $e$ of $G$. For each $l=1,\dots \lvert G_y\rvert$, the oscillatory integral
		\[\int_G (1-\chi_l(g))T_{P,-}(x,\,g\circ y)\,\mathrm{dV}_G(g)\equiv 0 \qquad \text{on } U\,.\]
\end{lem} 

We take $\chi_1$ to be a conjugate $G$-invariant bump function, that is, 
$\chi_1(ghg^{-1})=\chi_1(h)$, for every $g, h\in G$. The reason why the bump function can be taken to be conjugate $G$-invariant is as follows. If $\chi$ is a bump function in a neighborhood of the origin in the Lie algebra, then we can take a new bump function $$\eta(\exp_{\mathfrak{g}}{v})=\int_G\chi(g \exp_{\mathfrak{g}}{(v)} g^{-1})\mathrm{dV}_G(g)\,.$$
It is invariant under the action of conjugation. In order to prove that $\eta$ is a bump function one can see that $\eta(e)=1$ and $\eta$ is constant on a small ball center in the origin. In fact, the scalar product on the Lie algebra is $\mathrm{Ad}_g$-invariant, this means that $\lVert v \rVert=\lVert \mathrm{Ad}_g(v)\rVert$. Furthermore $\chi$ can be chosen to be constant on spherical shell, in particular is zero outside a ball of radius say $r$. This means that $\chi(\exp_{\mathfrak{g}}(\mathrm{Ad}_g(v)))$ is zero whenever $v$ is chosen so that $\lVert v\rVert>r$. This implies that also $\eta$ is zero outside a ball of radius $r$. 

Hence, we can write
\begin{equation} \label{eq:SqG0}
 T^G_{P,-}(x,\,y)= \sum_{l=1}^{\lvert G_y\rvert} \int_{V_l}\chi_l(g)\,T_{P,-}(x,\,g\circ y)\,\mathrm{dV}_G(g)\,. \end{equation}
Furthermore, notice that $V_l$ can be chosen so that $V_l=g_l\,V_1$, $l=2,\ldots,  \lvert G_y\rvert$, and
\[\chi_l(g)=\chi_1(g_l^{-1}\,g)\,, \]
we can change variables $g^{-1}_l\,g\mapsto g$ in \eqref{eq:SqG0} and obtain
\[T^G_{P,-}(x,\,y)\equiv\sum_{l=1}^{\lvert G_y\rvert} \int_{V_1}\chi_1(g)\,T_{P,-}(x,\,g_lg\circ y)\,\mathrm{dV}_G(g)\,. \]
Thus $T^G_{P,-}(x,\,y)$ is a sum of $\lvert G_y\rvert$ oscillatory integrals, let us denote each summand $T^G_{P,-}(x,\,y)_l$ and focus on one of them. 

Given $g=\exp_{\mathfrak{g}}{v}$, we have $g_lg\circ y= g_lgg_l^{-1}\circ y = \exp_{\mathfrak{g}}({\mathrm{Ad}_{g_l}(v)})\circ y$. Hence, we obtain
\begin{align*}
T^G_{P,-}(x,\,y)_l= \int_{V_1}\chi_1(\exp_{\mathfrak{g}}{v})\,T_{P,-}(x,\,\exp_{\mathfrak{g}}({\mathrm{Ad}_{g_l}(v)})\circ y)\,\mathrm{dV}_{\mathfrak{g}}(v) \end{align*} 
Notice that $\mathrm{Ad}_{g_l}\in GL(\mathfrak{g})$ and we can change variables $\mathrm{Ad}_{g_l}(v) \mapsto v$ in $\mathrm{d}v$. The absolute value of Jacobian determinant of the latter transformation is $\vert \det\mathrm{Ad}_{g_l} \rvert=1$, since the group is uni-modular. Thus, we get 
\begin{align*}
T^G_{P,-}(x,\,y)_l= \int_{V_1}\chi_1(\exp_{\mathfrak{g}}{v})\,T_{P,-}(x,\,\exp_{\mathfrak{g}}({v})\circ y)\,\mathrm{dV}_{\mathfrak{g}}(v) \end{align*}
where we use that $\chi_1$ is invariant under conjugation:
\[\chi_1(\exp_{\mathfrak{g}}\mathrm{Ad}_{g_l} (v))=\chi_1(g_l\exp_{\mathfrak{g}}{(v)}g_l^{-1})=\chi_1(\exp_{\mathfrak{g}}{v}) \,.\]
Thus each summand $T^G_{P,-}(x,\,y)_l$ does not depend on $l$, \eqref{e-gue210802yydII} can be deduced arguing as in the proof of Theorem $1.5$ in \cite{hsiaohuang} by making use the stationary phase formula of A. Melin and J. Sj\"ostrand.

To prove \eqref{e-gue210802yyd}, let $\chi_l$ be a bump function vanishing outside a small neighborhood $V_l$ of each $g_l\in G_p$ and $g_1$ is the identity $e$ of $G$. For each $l=1,\dots \lvert G_p\rvert$,  and we take $\chi_1$ to be a conjugate $G$-invariant bump function, $\chi_l(g)=\chi_1(g^{-1}_lg)$, $l=2,\ldots, \lvert G_p\rvert$. 
As \eqref{e-gue210804yyda}, we have 
\begin{equation}\label{e-gue210804yyda}
\begin{split}
&T^G_{P,-}(x,\,y)\equiv \sum_{l=1}^{\lvert G_p\rvert} \int_{V_1}\chi_1(g^{-1}_lg)\,T_{P,-}(x,\,g\circ y)\,\mathrm{dV}_G(g)\\
&\equiv\sum_{l=1}^{\lvert G_p\rvert} \int_{V_1}\chi_1(g^{-1}_lgg^{-1}_lg_l)\,T_{P,-}(x,\,g\circ y)\,\mathrm{dV}_G(g)\\
&\equiv\sum_{l=1}^{\lvert G_p\rvert} \int_{V_1}\chi_1(gg^{-1}_l)\,T_{P,-}(x,\,g\circ y)\,\mathrm{dV}_G(g)\\
&\equiv\sum_{l=1}^{\lvert G_p\rvert} \int_{V_1}\chi_1(g)\,T_{P,-}(x,\,gg_l\circ y)\,\mathrm{dV}_G(g).
\end{split}
\end{equation}
From \eqref{e-gue210804yyda}, we can repeat the proof of Theorem $1.5$ in \cite{hsiaohuang}  and get \eqref{e-gue210802yyd}.

Notice that, in order to compute the leading term 
\[a_-^0(p,\,p)=\frac{m(0)}{2}\,\pi^{-n-1+d/2}\,\lvert \mu_1\rvert^{\frac{1}{2}}\cdots \lvert \mu_d\rvert^{\frac{1}{2}}\cdot\lvert \mu_{d+1}\rvert\cdots \lvert \mu_n\rvert\,\tau_{p,\,n_-}\,, \]
one has to compute $m(0)$. For the locally free case we have
\[m(0)=2^{\frac{d}{2}}\frac{1}{V_{\mathrm{eff}}(p)\,\lvert G_p\rvert}\,. \]

From Theorem~\ref{t-gue210724yyd}, Theorem~\ref{t-gue210724yyda} and Theorem~\ref{t-gue210724yydc}, we can repeat the proof of Theorem~\ref{thm:toeplitz} and get Theorem~\ref{thm:composition} and Theorem~\ref{t-gue210803ycd}.

\section{Proof of Theorem~\ref{thm:kfourierszego}}

Let us consider the integral
\[T^G_{f,m}(x,\,y)=\frac{1}{2\,\pi}\int_{-\pi}^{\pi}\int_G T_f(x,\,e^{\imath\,\theta}\cdot g\circ y)\,e^{\imath m\,\theta} \mathrm{dV}_G(g)\,\mathrm{d}\theta\,. \]
It is easy to proof that the oscillatory integral has a rapidly decreasing asymptotic as $m\rightarrow +\infty$ far away from a local neighborhood of those elements $(e^{\imath\,\theta},\,g)\in S^1 \times G$ such that
\begin{equation} \label{eq:fix}
e^{\imath\,\theta}\cdot g\circ y=y\,. \end{equation}

We recall some notations. Since the action is free $\mu^{-1}(0)/S^1$ is an orbifold, let us denote with $\pi\,:\,\mu^{-1}(0)\rightarrow \mu^{-1}(0)/S^1$ the projection. Furthermore the action of $S^1$ commutes with the one of $G$, then we have a smooth locally free action of $G$ on $\mu^{-1}(0)/S^1$. Given $y\in X$ and $g\in G_{\pi(y)}$, there exist $\lvert S_x^1 \rvert$ elements $e^{\imath\,\theta_{g,j}}\in S^1$ ($j=1,\dots,\lvert S_x^1 \rvert$) such that 
\[g\circ y= e^{-\imath\,\theta_{g,j}}\circ y\,. \]
Thus, all the elements $(e^{\imath\,\theta},\,g)\in S^1 \times G$ satisfying \eqref{eq:fix} are of the form
$e^{\imath \theta_{g,j}}\cdot g $ for each $g\in G_{\pi(y)}$. Thus, 
\begin{align*} 
	T^G_{f,\,m}(x,\,y)=&\frac{1}{2\,\pi}\sum_{h\in G_{\pi(y)}}\sum_{j=1}^{\lvert S_x^1 \rvert}\int_{-\delta}^{\delta}\mathrm{d}\theta\int_{V_0}\mathrm{d}v \left[ T_f(x,\,e^{\imath(\theta_{h,j}+\theta)}\cdot h\,e^{\imath\,v}\circ y)\,e^{\imath m(\theta_{h,j}+\theta)} m(v)\right]
\end{align*}
where $\delta>0$, $V_0$ is a small neighborhood of the identity (we omit the bump functions for ease of exposition). We use the same trick as in the proof of Theorem~\ref{thm:toeplitz}:
\[e^{\imath\theta_h}\cdot h\,e^{\imath\,v}h^{-1}\,h\circ y= h\,e^{\imath\,v}h^{-1}\circ y= \,e^{\imath\,\mathrm{Ad}_h(v)}\circ y\,; \]
changing variables $v\mapsto \mathrm{Ad}_h(v)$ we obtain
\begin{align*}
T^G_{f,\,m}(x,\,y)=&\frac{1}{2\,\pi}\sum_{h\in G_{\pi(y)}}\sum_{j=1}^{\lvert S_x^1 \rvert}e^{\imath m\,\theta_{h,j}}\,\int_{-\delta}^{\delta}\mathrm{d}\theta\,\int_{V_0} \mathrm{d}v\left[ T_f(x,\, e^{\imath \,\theta}\cdot e^{\imath\,v}\circ y)\,e^{\imath m\theta} m(v)\right]\,. \end{align*}
The oscillatory integral appearing in the previous formula was already studied in \cite{hsiaohuang} for the free case (see especially the discussion after equation $(4.16)$ in \cite{hsiaohuang}).  

\section{Proof of Theorem \ref{thm:compositionFouriercomponents}}

We can repeat the proof of Lemma~\ref{l-gue210624yyd} with minor change and deduce 

\begin{lem}\label{l-gue210812yyd} 
Let $B\in\Psi_k(X)_G$ with $S^{(q)}_GB\equiv B\equiv BS^{(q)}_G$. Let $x=(x_1,\ldots,x_{2n+1})$ be local coordinates of $X$ defined on an open set $D\subset X$. 
Assume that 
\[\mbox{$\tau_{x,n_-}\sigma^0_{B,-}(x,x)\tau_{x,n_-}=0$, for every $x\in D$},\]
and 
\[\mbox{$\tau_{x,n_+}\sigma^0_{B,+}(x,x)\tau_{x,n_+}=0$, for every $x\in D$},\]
if $q=n_+$. Then, $B\in\Psi_{k-1}(X)_G$. 
\end{lem}

\begin{lem}\label{l-gue210812yydI}
Let $B\in\Psi_n(X)_G$ with $S^{(q)}_GB\equiv B\equiv BS^{(q)}_G$. Then, 
\[B: L^2_{(0,q)}(X)\To L^2_{(0,q)}(X)\]
is continuous. 
\end{lem} 

\begin{proof}
From complex stationary phase formula, we can find $f\in\mathcal{C}^\infty(X)^G$ such that 
\begin{equation}\label{e-gue210812yyd}
\tau_{x,n_-}\sigma^0_{T_f,-}(x,x)\tau_{x,n_-}=\sigma^0_{B,-}(x,x),\ \ \mbox{for every $x\in X$}.
\end{equation}
From Lemma~\ref{l-gue210812yyd} and \eqref{e-gue210812yyd}, we conclude that 
\begin{equation}\label{e-gue210812yyda}
R:=T^G_{f,-}-B_-\in\Psi_{n-1}(X)_G.
\end{equation}
We claim that 
\begin{equation}\label{e-gue210812yydb}
\mbox{$R: L^2_{(0,q)}(X)\To L^2_{(0,q)}(X)$ is bounded}. 
\end{equation}
Let $u\in\Omega^{0,q}(X)$. We have 
\begin{equation}\label{e-gue210812ycd}
\norm{Ru}^2\leq\norm{(R^*R)u}\norm{u}\leq\cdots\leq\norm{(R^*R)^{2^N}}^{2^{-N}}\norm{u}^{2-2^{-N}},
\end{equation}
for every $N\in\mathbb N$, where $R^*$ is the adjoint of $R$. From stationary phase formula, we can check that 
\[(R^*R)^{2^N}\in\Psi_{n-2^{N+1}}(X)_G.\]
We take $N$ large enough so that $(R^*R)^{2^N}$ is $L^2$ bounded. From this observation and \eqref{e-gue210812ycd}, we get the claim \eqref{e-gue210812yydb}. 
From \eqref{e-gue210812yyda}, \eqref{e-gue210812yydb} and notice that $T^G_{f,-}$ is $L^2$ bounded, the lemma follows. 
\end{proof}

\begin{proof}[Proof of Theorem~\ref{thm:kfourierszego}]
We will use the same notations as Section~\ref{s-gue210812yyd}. For $f, g\in\mathcal{C}^\infty(X)^G_0$, let $C_j(f,g):=\hat C_{j,-}(f,g)$, $j=0,1,2,\ldots$, where 
$\hat C_{j,-}(f,g)$, $j=0,1,2,\ldots$, are as in Theorem~\ref{t-gue210706ycdg}. It is straightforward to see that $C_j(f,g)\in\mathcal{C}^\infty(X)^G_0$, $j=0,1,\ldots$. 
Fix $N\in\mathbb N$. From \eqref{e-gue210706ycds}, we have 
\begin{equation}\label{e-gue210812yyds}
\hat R^{N+1}T^G_{f,-}\circ T^G_{g,-}-\sum^N_{j=0}\hat R^{N+1-j}T^G_{C_{j}(f,g),-}\in\Psi_{n}(X)_G.
\end{equation}
From Lemma~\ref{l-gue210812yydI} and \eqref{e-gue210812yyds}, there is a constant $C>0$ such that 
\begin{equation}\label{e-gue210812yydt}
\norm{\Bigr(\hat R^{N+1}T^G_{f,-}\circ T^G_{g,-}-\sum^N_{j=0}\hat R^{N+1-j}T^G_{C_{j}(f,g),-}\Bigr)u}\leq C\norm{u},
\end{equation}
for every $u\in\Omega^{0,q}(X)$. From \eqref{e-gue210812yydt}, we get 
\begin{equation}\label{e-gue210812yydu}
\norm{\Bigr(m^{N+1}T^G_{f,m}\circ T^G_{g,m}-\sum^N_{j=0}m^{N+1-j}T^G_{C_{j}(f,g),m}\Bigr)u}\leq C\norm{u},
\end{equation}
for every $u\in\Omega^{0,q}(X)$. From \eqref{e-gue210812yydu}, we get \eqref{e-gue210803ycdt}. The proof of \eqref{e-gue210803ycds} is similar. 
\end{proof}

\textbf{Acknowledgements:} This project was started during the first author’s postdoctoral fellowship at the National Center for Theoretical Sciences in Taiwan; we thank the Center for the support. Chin-Yu Hsiao was partially supported by Taiwan Ministry of Science and Technology projects  108-2115-M-001-012-MY5, 109-2923-M-001-010-MY4.

\end{document}